\documentclass{amsart}%
\usepackage{amssymb}
\usepackage{amsfonts}
\usepackage{amsmath}
\usepackage{graphicx}
\usepackage{subfigure}%
\providecommand{\U}[1]{\protect\rule{.1in}{.1in}}
\pdfoutput=1
\newtheorem{theorem}{Theorem}
\theoremstyle{plain}

\newtheorem{condition}{Condition}

\newtheorem{corollary}{Corollary}

\newtheorem{lemma}{Lemma}

\newtheorem{proposition}{Proposition}
\newtheorem{remark}{Remark}

\numberwithin{equation}{section}
\begin{document}
\title{Existence of travelling pulses in a neural model}
\author{S. P. Hastings.}

\begin{abstract}
In 1992 G. B. Ermentrout and J. B. McLeod published a landmark study of
travelling wavefronts for a differential-integral equation model of a neural
network. Since then a number of authors have extended the model by adding an
additional equation for a \textquotedblleft recovery
variable\textquotedblright, thus allowing the possibility of travelling pulse
type solutions. In a recent paper G. Faye gave perhaps the first rigorous
proof of the existence (and stability) of a travelling pulse solution for a
model of this type, treating a simplified version of equations originally
developed by Kilpatrick and Bressloff. The excitatory weight function $J$ used
in this work allowed the system to be reduced to a set of four coupled ODEs,
and a specific firing rate function $S$, with parameters, was considered. The
method of geometric singular perturbation was employed, together with
blow-ups. \ In this paper we extend Faye's results on existence by dropping
one of his key hypotheses, proving the existence of pulses at at least two
different speeds, and in a sense, allowing a wider range of the small
parameter in the problem. \ The proofs are classical, and self-contained aside
from standard ode material.

\end{abstract}
\dedicatory{This paper is dedicated to the memory of J. Bryce McLeod, 1929-2014, a dear
friend and inspiring collaborator. }\maketitle

\section{Introduction}

In this paper we consider the following system of equations:%
\begin{equation}
\left.
\begin{array}
[c]{c}%
\frac{\partial u\left(  x,t\right)  }{\partial t}=-u\left(  x,t\right)
+\int_{-\infty}^{\infty}J\left(  x-y\right)  q\left(  y,t\right)  S\left(
u\left(  y,t\right)  \right)  dy\\
\frac{1}{\varepsilon}\frac{\partial q\left(  x,t\right)  }{\partial
t}=1-q\left(  x,t\right)  -\beta q\left(  x,t\right)  S\left(  u\left(
x,t\right)  \right)
\end{array}
\right.  ,\ \label{-1}%
\end{equation}
where $J$ is a normalized exponential
\begin{equation}
J\left(  x\right)  =\frac{b}{2}e^{-b\left\vert x\right\vert } \label{-2}%
\end{equation}
and the \textquotedblleft firing rate\textquotedblright\ function $S$ is given
by
\begin{equation}
S\left(  u\right)  =\frac{1}{1+e^{\lambda\left(  \kappa-u\right)  }},
\label{-3}%
\end{equation}
for certain positive parameters $\varepsilon,$ $\lambda,$ $b,$ $\kappa,$ and
$\beta.$ The variable $u$ is the synaptic input current for a neural network
with synaptic depression, the effect of which is represented by the scaling
factor $q.$ These equations were proposed and studied by G. Faye in
\cite{faye}.

The Faye model is a simplified version of one first introduced by Kilpatrick
and Bressloff in \cite{kilpatrickbressloff}. These authors included a variable
and equation to allow for spike frequency adaptation. However they show by
numerical computation that adaptation has little effect on the resulting
waves. Faye dropped the adaptation equation and variable in
\cite{kilpatrickbressloff} to get his system (\ref{-1}). \ See \cite{faye} and
\cite{kilpatrickbressloff} for further information on the physical background
of (\ref{-1}).

\bigskip

In \cite{faye} the author proves two interesting results about the system
(\ref{-1}), namely the existence of a \textquotedblleft travelling
pulse\textquotedblright\ solution and the stability of this solution. A
travelling pulse solution of (\ref{-1}) is a non-constant solution $\left(
u,q\right)  $ of the form
\[
\left(  u\left(  x+ct\right)  ,q\left(  x+ct\right)  \right)
\]
such that both $\lim_{s\rightarrow\infty}\left(  u\left(  s\right)  ,q\left(
s\right)  \right)  $ and $\lim_{s\rightarrow-\infty}\left(  u\left(  s\right)
,q\left(  s\right)  \right)  $ exist and these limits are equal. In this paper
we are interested in the existence of values of $c$ for which (\ref{-1}) has
such a solution. As we describe briefly below, using (\ref{-2}) leads to a set
of four ode's in which $c$ is a parameter. To show that a travelling pulse
exists for some $c>0$, Faye uses the theory of geometric singular perturbation
initiated by Fenichel in \cite{fenichel} and extended by Jones and Kopell in
\cite{jones}. The \textquotedblleft blowup\textquotedblright\ method is also
employed \cite{dumortier}.

Here we extend the existence result in \cite{faye} in several ways. We show
that for sufficiently small $\varepsilon>0$ there are at least two travelling
pulses, hence a \textquotedblleft fast\textquotedblright\ pulse and a
\textquotedblleft slow\textquotedblright\ pulse, for speeds $c^{\ast}>c_{\ast
}>0$. Also, we remove an important hypothesis used in \cite{faye}, one which
can only be verified by numerical integration of a related ode
system.\ \ (This hypothesis is stated and discussed in Section
\ref{discussion}.) \ Our proof is for a general class of firing functions
which includes the specific $S$ for which Faye states his theorem.

Further, we use a method which allows, in some sense, a larger range of
$\varepsilon$ than seems possible with geometric perturbation. This will be
made precise in the statements of our theorems. We believe, based on our past
experience with a similar problem, that it is feasible to check existence
rigorously for particular positive values of $\varepsilon>0,$ using precise
numerical analysis based on interval arithmetic, but we have not carried out
such a check.\ This will be explained further in Section \ref{discussion}.

We now mention two well-known predecessors of the Kilpatrick-Bressloff and
Faye models. \ In 1992, Ermentrout and McLeod studied the equation
\begin{equation}
u_{t}=-u+\int_{-\infty}^{\infty}J\left(  x-y\right)  S\left(  u\left(
t,y\right)  \right)  dy. \label{emc}%
\end{equation}
As above, $S$ is positive, bounded, and increasing. \ Since there is no
feedback in the equation, (\ref{emc}) supports only traveling fronts, where
$u$ is monotone. In the landmark paper \cite{emc} Ermentrout and McLeod proved
the existence of fronts for a wide variety of symmetric positive weight
functions $J$ and firing rates $S.$ \ (Their work applied to a more general
equation )\ Subsequently, in \cite{pe}, Pinto and Ermentrout introduced the
needed negative feedback in order to get pulses. Their system is%
\begin{equation}
\left.
\begin{array}
[c]{c}%
\frac{\partial u\left(  x,t\right)  }{\partial t}=-u-v+\int_{-\infty}^{\infty
}J\left(  x-y\right)  S\left(  u\left(  t,y\right)  \right)  dy\\
\frac{1}{\varepsilon}\frac{\partial q\left(  x,t\right)  }{\partial
t}=u-\gamma v
\end{array}
\right.  . \label{pe}%
\end{equation}
They analyzed this system primarily for the case $S\left(  u\right)  =H\left(
u-\kappa\right)  $ where $H$ is the Heaviside function and $\kappa$ is a
constant representing a firing threshold. \ While some partial results have
been obtained recently by Scheel and Faye (see Section \ref{discussion}), we
are not aware of any existence proof for pulses which covers all reasonable
smooth functions $S.$ We discuss what we mean by \textquotedblleft
reasonable\textquotedblright\ in Section \ref{discussion}, where we also
indicate why our method does not appear to apply to this model, and why we
expect that (\ref{pe}) supports a richer family of bounded traveling waves
than exist for (\ref{-1}). \footnote{We thank the anonymous referee, who
helped improve the presentation and who asked several questions which should
have been answered in the original version.}\ \ 

\section{Statement of results}

Travelling pulse solutions of (\ref{-1}) with (\ref{-2}) are shown to satisfy
a system of ode's by letting $v\left(  s\right)  =\int_{-\infty}^{\infty}%
\frac{b}{2}e^{-b\left\vert s-\tau\right\vert }q\left(  \tau\right)  J\left(
u\left(  \tau\right)  \right)  d\tau$ and computing $w=v^{\prime}$ and
$w^{\prime}$.\footnote{It is not necessary to discuss Fourier transforms, as
is usually done here.} We find that
\begin{equation}
\left.
\begin{array}
[c]{c}%
u^{\prime}=\frac{v-u}{c}\\
v^{\prime}=w\\
w^{\prime}=b^{2}\left(  v-qS\left(  u\right)  \right) \\
q^{\prime}=\frac{\varepsilon}{c}\left(  1-q-\beta qS\left(  u\right)  \right)
.
\end{array}
\right.  \label{1}%
\end{equation}
We will denote solutions of this system by $p=\left(  u,v,w,q\right)  ,$ and
we look for values of $c$ for which there is a non-constant solution $p$ such
that $p\left(  \infty\right)  $ and $p\left(  -\infty\right)  $ both exist and
are equal. The orbit of such a solution of (\ref{1}) is called
\textquotedblleft homoclinic\textquotedblright. In the language of dynamical
systems, $\left(  u\left(  x+ct\right)  ,q\left(  x+ct\right)  \right)  $ is a
pulse solution of (\ref{-1}) if and only if the orbit of $p$ is homoclinic.

We make the following assumptions on $S$.

\begin{condition}
\label{c0} The function $S$ is positive, increasing, bounded, and has a
continuous first derivative $S^{\prime}.$
\end{condition}

\begin{condition}
\label{c0a}The function $h\left(  u\right)  =\frac{u}{S\left(  u\right)  }$
has one local maximum followed by one local minimum, and no other critical points.
\end{condition}

\begin{condition}
\label{c1}$S$ is such that the system (\ref{1}) has exactly one equilibrium
point, say $p_{0}=\left(  u_{0},u_{0},0,q_{0}\right)  $.
\end{condition}

\begin{condition}
\label{c2}The function $S$ is also such that the \textquotedblleft
fast\textquotedblright\ system%
\begin{equation}
\left.
\begin{array}
[c]{c}%
u^{\prime}=\frac{v-u}{c}\\
v^{\prime}=w\\
w^{\prime}=b^{2}\left(  v-q_{0}S\left(  u\right)  \right)
\end{array}
\right.  \label{2}%
\end{equation}
has three equilibrium points, $\left(  u_{0},u_{0},0\right)  ,$ $\left(
u_{m},u_{m},0\right)  $, and $\left(  u_{+},u_{+},0\right)  ,$ with
$u_{0}<u_{m}<u_{+}.$
\end{condition}

\begin{condition}
\label{c3}%
\[
\int_{u_{0}}^{u_{+}}\left(  q_{0}S\left(  u\right)  -u\right)  du>0.
\]

\end{condition}

For convenience we will assume that $0<S<1$ on $\left(  -\infty,\infty\right)
.$ Then Conditions \ref{c0}-\ref{c2} imply that $0<q_{0}<1,$ $u_{0}>0$, and
$u_{+}<1.$

We will denote solutions of (\ref{2}) by $r=\left(  u,v,w\right)  .$ The local
minimum of $h$ will be denoted by $u_{knee}.$ In \cite{faye} specific ranges
of $\kappa$ and $\lambda$ are given so that these conditions are satisfied by
the function given in (\ref{-3}). In Figure \ref{figurea} we show the graphs
of $h$, $\frac{1}{1+\beta S}$ (the $q$ nullcline), and $q=q_{0}$, when $S$ is
given by (\ref{-3}). We use the same parameter values as were chosen for
illustration in \cite{faye}.\footnote{$\lambda=20,\ \kappa=0.22,\beta
=5,\ b=4.5$}

\begin{figure}[h]
\includegraphics[height=1.5 in, width =2  in]{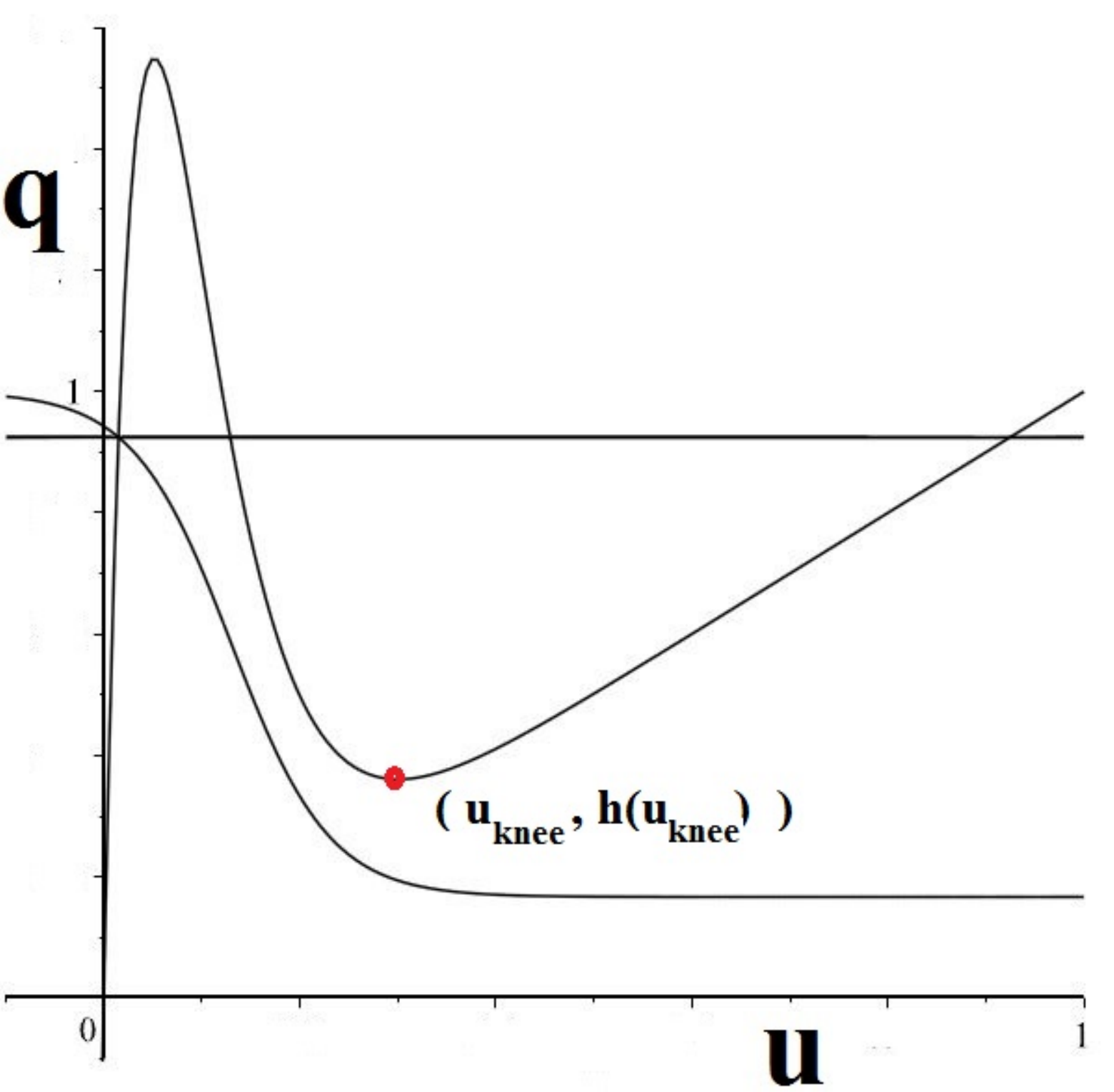}\caption{graphs of
$h$, $\frac{1}{1+\beta S}$ , and $q=q_{0}$ }%
\label{figurea}%
\end{figure}

We can now state our first main result.

\begin{theorem}
\label{thm1a}If Conditions \ref{c0}- \ref{c3} are satisfied, and $\varepsilon$
is positive and sufficiently small, then there are at least two positive
values of $c,$ say $c^{\ast}>c_{\ast},$ such that (\ref{1}) has a non-constant
solution $p$ satisfying
\[
\lim_{t\rightarrow-\infty}p\left(  t\right)  =\lim_{t\rightarrow\infty
}p\left(  t\right)  =p_{0}.
\]

\end{theorem}

In order to state our remaining theorems it is convenient first to give some
basic information about the fast system, (\ref{2}). We state this information
as a pair of lemmas, which will be used in proving our Theorems. Their proofs
are given in the Appendix.

\begin{lemma}
\label{lem1}If Conditions \ref{c2} and \ref{c3} are satisfied, then for each
$c>0$ the equilibrium point $\left(  u_{0},u_{0},0\right)  $ of (\ref{2}) is a
saddle point, with a one dimensional unstable manifold $\mathcal{U}_{0,c}$ and
a two dimensional stable manifold $\mathcal{S}_{0,c}.$ There is, for each
$c>0,$ a unique solution $r_{c}=\left(  u_{0,c},v_{0,c},w_{0,c}\right)  $ of
(\ref{2}) with $r_{c}\left(  t\right)  \in\mathcal{U}_{0,c}$ for all $t$ and
satisfying the conditions
\begin{equation}
\left.
\begin{array}
[c]{c}%
u_{0,c}\left(  0\right)  =u_{m}\\
w_{0,c}>0\text{ on }(-\infty,0].
\end{array}
\right.  \label{3}%
\end{equation}
Further, there is a unique $c=c_{0}^{\ast}>0$ such that $w_{0,c_{0}^{\ast}}>0$
on $R$ and
\[
\lim_{t\rightarrow\infty}r_{c_{0}^{\ast}}\left(  t\right)  =\left(
u_{+},u_{+},0\right)  .
\]

\end{lemma}

In other words, the branch $\mathcal{U}_{0,c_{0}^{\ast}}^{+}$ of
$\ \mathcal{U}_{0,c_{0}^{\ast}}$ pointing into the positive octant
$u>u_{0},v>u_{0},w>0$ \ is a heteroclinic orbit connecting $\left(
u_{0},u_{0},0\right)  $ to $\left(  u_{+},u_{+},0\right)  .$ Also,
$~w_{0,c_{0}^{\ast}}>0$ on $\left(  -\infty,\infty\right)  $, which implies
that $v_{0,c_{0}^{\ast}}^{\prime}>0$ and $u_{0,c_{0}^{\ast}}^{\prime}>0$.
\ This solution is called a \textquotedblleft front\textquotedblright\ for
(\ref{2}). A front for (\ref{2}) can be characterized as a solution of this
equation which exists on $\left(  -\infty,\infty\right)  ,$ is nonconstant and
bounded, and satisfies $u_{c}^{\prime}>0$ on $\left(  -\infty,\infty\right)
.$

\begin{lemma}
\label{lem2}If $c>c_{0}^{\ast}$ then $w_{0,c}>0$ on $R$, and
\[
\lim_{t\rightarrow\infty}u_{0,c}\left(  t\right)  =\lim_{t\rightarrow\infty
}v_{0,c}\left(  t\right)  =\infty.
\]
If $0<c<c_{0}^{\ast}$ then $w_{0,c}$ is initially positive and has a unique
zero. Also, $u_{0,c}^{\prime}$ has a unique zero, and
\[
\lim_{t\rightarrow\infty}u_{0,c}\left(  t\right)  =\lim_{t\rightarrow\infty
}v_{0,c}\left(  t\right)  =-\infty.
\]
If $c\in\left(  0,c_{0}^{\ast}\right)  $ and $t_{1}\left(  c\right)  $ is the
zero of $u_{0,c}^{\prime}$, where $u_{0,c}$ is a maximum, then $u_{0,c}%
^{\prime\prime}\left(  t_{1}\left(  c\right)  \right)  <0$ and $\lim
_{c\rightarrow c_{0}^{\ast-}}r_{c}\left(  t_{1}\left(  c\right)  \right)
=\left(  u_{+},u_{+},0\right)  .$

Suppose finally that for some $c_{1}\in\left(  0,c_{0}^{\ast}\right)  ,$
$w_{0,c_{1}}>0$ on an interval $\left(  -\infty,T\left(  c_{1}\right)
\right)  ,$ $w_{0,c_{1}}\left(  T\left(  c_{1}\right)  \right)  =0,$ and
\[
v_{0,c_{1}}\left(  T\left(  c_{1}\right)  \right)  >q_{0}S\left(
u_{knee}\right)  .
\]
Then for any $c\in\lbrack c_{1},c_{0}^{\ast})$,
\[
u_{0,c}\left(  t_{1}\left(  c\right)  \right)  >u_{knee}.
\]

\end{lemma}

\begin{remark}
\label{rem1}We conjecture that the condition
\[
u_{0,c_{1}}\left(  T\left(  c_{1}\right)  \right)  >u_{knee}%
\]
would imply the same conclusion, but we have not been able to prove this.
\bigskip
\end{remark}

The positive number $c_{0}^{\ast}$ defined in Lemma \ref{lem1} plays an
important role throughout this paper.

\begin{theorem}
\label{thm1}Suppose that Conditions \ref{c0}- \ref{c3} are satisfied. Suppose
also that there is a $c_{1}\in\left(  0,c_{0}^{\ast}\right)  $, such that if
$T\left(  c_{1}\right)  $ is the unique zero of $w_{0,c_{1}}$ (which exists by
Lemma \ref{lem2}), then
\begin{equation}
v_{0,c_{1}}\left(  T\left(  c_{1}\right)  \right)  >q_{0}S\left(
u_{knee}\right)  . \label{y1}%
\end{equation}

Assume as well that for some $\varepsilon>0$ there is a solution
$p_{\varepsilon,c_{1}}=\left(  u_{\varepsilon,c_{1}},v_{\varepsilon,c_{1}%
},w_{\varepsilon,c_{1}},q_{\varepsilon,c_{1}}\right)  $ of (\ref{1}) with
$c=c_{1}$ which has the following properties:%
\begin{equation}
\lim_{t\rightarrow-\infty}p_{\varepsilon,c_{1}}\left(  t\right)  =p_{0}
\tag{i}\label{i}%
\end{equation}%
\begin{equation}
\left\{
\begin{array}
[c]{c}%
u_{\varepsilon,c_{1}}^{\prime}>0\text{ on some interval \thinspace}%
(-\infty,t_{1})\text{ and }u_{\varepsilon,c_{1}}^{\prime\prime}\left(
t_{1}\right)  <0\\
u_{\varepsilon,c_{1}}^{\prime}<0\text{ on some interval }(t_{1},t_{3}]\text{
and }u_{\varepsilon,c_{1}}\left(  t_{3}\right)  =0
\end{array}
\right.  \tag{ii}\label{ii}%
\end{equation}

Then for the given $\varepsilon$ there are two values of $c,$ say $c_{\ast}%
\in\left(  0,c_{1}\right)  $ and $c^{\ast}\in\left(  c_{1},c_{0}^{\ast
}\right)  $ such that (\ref{1}) has a homoclinic orbit.
\end{theorem}

Figure \ref{fig2a} below includes a graph of the orbit of a solution
satisfying (\ref{i}) and (\ref{ii}) projected onto the $\left(  u,q\right)  $
plane, with the points $t_{1}$ and $t_{3}$ marked (as well as an additional
point $t_{2}$ which is explained later). The other solution shown in that
figure satisfies (\ref{i}) but not (\ref{ii}).

Theorem \ref{thm1a} is implied by Theorem \ref{thm1} and the following result.

\begin{theorem}
\label{thm2}If Conditions \ref{c0}- \ref{c3} are satisfied then there is a
$c_{1}$ satisfying the conditions in the second sentence of Theorem
\ref{thm1}. Further, with this $c_{1},$ if $\varepsilon$ is sufficiently
small, then the solution $p_{\varepsilon,c_{1}}$ of (\ref{1}) with $c=c_{1}$
satisfies (\ref{i}) and (\ref{ii}) of Theorem \ref{thm1}.
\end{theorem}

It will follow from the proofs of these results that as $\varepsilon
\rightarrow0,$ $c^{\ast}\rightarrow c_{0}^{\ast}$. \ The following result is
all we have proved about the asymptotic behavior of $c_{\ast}.$

\begin{theorem}
\label{thm3}%
\[
\lim_{\varepsilon\rightarrow0}c_{\ast}=0.
\]
However there is an $M>0$ independent of $\varepsilon$ such that if there is a
homoclinic orbit for $c=c_{\ast}$ then $\frac{\varepsilon}{c_{\ast}}<M.$
\end{theorem}

\begin{remark}
For a given pair $\left(  \varepsilon,c_{1}\right)  $, the hypotheses of
Theorem \ref{thm1} can be verified by checking one solution of (\ref{2}) at
$c=c_{1}$ and one solution of (\ref{1}), with the given $\varepsilon$ and
$c=c_{1}.$ For the specific model considered in \cite{faye}, standard
numerical analysis (non-rigorous) easily finds specific values of $\left(
\varepsilon,c_{1}\right)  $ where these hypotheses are apparently
satisfied\footnote{\bigskip For the parameter values used by Faye, a standard
ode solver suggests that $\left(  \varepsilon,c_{1}\right)  =\left(
.005,.34\right)  $ \ satisfies the conditions in Theorem \ \ref{thm1}. If the
conjecture in Remark \ref{rem1} is true then it appears that $\left(
\varepsilon,c_{1}\right)  =\left(  .05,.2\right)  $ would work.}. In the
discussion section we describe how this could, in principle, be checked
rigorously using uniform asymptotic analysis near the equilibrium points of
(\ref{2}) and \ref{1}), and then a rigorous numerical ode solver (using
interval arithmetic) over two compact intervals. We cite a paper\footnote{(on
the Lorenze equations)} where a similar procedure was followed successfully,
but we have not attempted it here.
\end{remark}

\begin{remark}
In \cite{faye} only one homoclinic solution is found, and there is an extra
hypothesis about the system (\ref{2}). (Hypotheses 3.1) \ As far as we know,
this hypothesis can only be checked by numerically solving the system
(\ref{2}).\footnote{It apears to us that because of the degeneracy at the
knee, it would be harder to verify this hypothesis rigorously than to do the
same for (i) and (ii).} \ We discuss this further in Section \ref{discussion}.
\end{remark}

\section{Proof of Theorem \ref{thm1}}

\subsection{The fast pulse}

We need two simple preliminary results about the behavior of solutions.

\begin{proposition}
\label{prop1}For any $\varepsilon\geq0$ the regions $\left\{  v<0,w<0,\frac
{1}{1+\beta}<q<1\right\}  $ and

$\left\{  v>1,\ w>0,\frac{1}{1+\beta}<q<1\right\}  $\ are positively invariant
open sets for the system (\ref{1}).
\end{proposition}

\begin{proof}
We are assuming that $0<S\left(  u\right)  <1$ for all $u.$ Hence, $q^{\prime
}>0$ if $q\leq\frac{1}{1+\beta}$ and $q^{\prime}<0$ if $q\geq1.$ Therefore
$\left\{  \frac{1}{1+\beta}<q<1\right\}  $ is positively invariant. Further,
if $\frac{1}{1+\beta}<q<1$ then $v^{\prime\prime}=w^{\prime}<0$ if $v\leq0$
and $w^{\prime}>0$ if $v\geq1.$ The result follows.
\end{proof}

Note as well that because $S$ is bounded, all solutions of (\ref{1}) exist on
$R=\left(  -\infty,\infty\right)  .$

\begin{proposition}
\label{prop0} If $p=\left(  u,v,w,q\right)  $ is a solution of (\ref{1}), and
$u\left(  t\right)  \geq u_{knee}$ for some $t,$ then either $q^{\prime
}\left(  t\right)  <0$ or $q\left(  t\right)  <h\left(  u_{knee}\right)  .$
\end{proposition}

\begin{proof}
This follows from Condition \ref{c1}, which implies that the graph of the
decreasing function $q=\frac{1}{1+\beta S\left(  u\right)  }$ in the $\left(
q,u\right)  $ plane, where $q^{\prime}=0$, passes under the point $\left(
u_{knee},h\left(  u_{knee}\right)  \right)  $. (See Figure \ref{figurea}.)
\end{proof}

In the first, and longest, part of the proof of Theorem \ref{thm1} we show
that there is a \textquotedblleft fast\textquotedblright\ pulse, with speed
$c^{\ast}(\varepsilon)$ which tends to $c_{0}^{\ast}$ as $\varepsilon$ tends
to zero. In the second part we look for a \textquotedblleft
slow\textquotedblright\ pulse, with a speed $c_{\ast}(\varepsilon)$ which
tends to zero as $\varepsilon$ tends to zero.

We will show that for any possible homoclinic orbit, $u>0$. We look for
homoclinic orbits such that, as well, $q<q_{0}$ in $(-\infty,\infty)$. In
searching for the fast solution we will consider for each $c>0$ a certain
uniquely defined solution $p_{c}=(u_{c},v_{c},w_{c},q_{c})$ such that
$p_{c}(-\infty)=p_{0}$. We will show that there is a nonempty bounded set of
positive values of $c$, called $\Lambda(\varepsilon)$, such that, among other
properties of $p_{c}$, either $q_{c}$ exceeds $q_{0}$ at some point, or
$u_{c}$ becomes negative. We then examine the behavior of $p_{c^{\ast
}(\varepsilon)}$ where $c^{\ast}(\varepsilon)=\sup\Lambda(\varepsilon)$. The
goal is to show that $p_{c^{\ast}(\varepsilon)}(\infty)=p_{0}$. This is done
be eliminating all the other possible behaviors of $p_{c^{\ast}(\varepsilon)}%
$, often by showing that a particular behavior implies that all values of $c$
close to $c^{\ast}(\varepsilon)$ are not in $\Lambda(\varepsilon)$.

The following result is basic to our analysis of the full system (\ref{1}).
The proof is routine and again left to the appendix.

\begin{lemma}
\label{lem2b} Suppose that Conditions \ref{c0}- \ref{c3} hold, and let
$p_{0}=\left(  u_{0},u_{0},0,q_{0}\right)  $ be the unique equilibrium point
of (\ref{1}). Then for any $\varepsilon\geq0$ and $c>0$ the system (\ref{1})
has a one dimensional unstable manifold at $p_{0}$, say $\mathcal{U}%
_{\varepsilon,c},$ with branch $\mathcal{U}_{\varepsilon,c}^{+}$ starting in
the region $\left\{  u>u_{0},v>u_{0},w>0,\right\}  .$ If $p_{\varepsilon
,c}=\left(  u_{\varepsilon,c},v_{\varepsilon,c},w_{\varepsilon,c}%
,q_{\varepsilon,c}\right)  $ is a solution lying on this manifold, then for
large negative $t,$ $u_{0}<u_{\varepsilon,c}\left(  t\right)  <v_{\varepsilon
,c}\left(  t\right)  $ and $w_{\varepsilon,c}\left(  t\right)  >0.$ Also,
$q_{0,c}\equiv q_{0},$ while if $\varepsilon>0$ then $q_{\varepsilon,c}\left(
t\right)  <q_{0}$ for large negative $t.$ The invariant manifold
$\mathcal{U}_{\varepsilon,c}^{+}$ depends continuously on $\left(
\varepsilon,c\right)  $ in $\varepsilon\geq0,$ $c>0.$ \ (The meaning of
continuity here is made clear in the text below.) Finally, if $\lambda
_{1}\left(  c,\varepsilon\right)  $ is the positive eigenvalue of the
linearization of (\ref{1}) around $p_{0},$ then $\lambda_{1}\left(
c,\varepsilon\right)  >\lambda_{1}\left(  c,0\right)  $ for each $c>0$ and
$\varepsilon>0.$
\end{lemma}

The following proposition follows trivially from (\ref{1}) and will be used a
number of times, often without specific mention.

\begin{proposition}
\label{prop2}
\begin{align*}
\text{\textit{If }}u^{\prime}  &  =0\ \text{\textit{then }}u^{\prime\prime
}=\frac{v^{\prime}}{c}.\\
\text{If }u^{\prime}  &  =u^{\prime\prime}=0\text{ then }u^{\prime\prime
\prime}=\frac{v^{\prime\prime}}{c}.\\
\text{If }u^{\prime}  &  =u^{\prime\prime}=u^{\prime\prime\prime}=0\text{ then
}u^{iv}=-\frac{b^{2}}{c}q^{\prime}S\left(  u\right)  .\\
\ \text{If }q^{\prime}  &  =0\text{\ then }q^{\prime\prime}=-\frac
{\varepsilon}{c}\beta qS^{\prime}\left(  u\right)  u^{\prime}\\
\text{If }w^{\prime}  &  =0\text{ then }v^{\prime\prime\prime}=w^{\prime
\prime}=b^{2}\left(  v^{\prime}-q^{\prime}S\left(  u\right)  -qS^{\prime
}\left(  u\right)  u^{\prime}\right)  .\\
\text{If }q^{\prime}  &  =u^{\prime}=0\text{\ then }q^{\prime\prime}=0\text{
and }q^{\prime\prime\prime}=-\frac{\varepsilon}{c}\beta qS^{\prime}\left(
u\right)  u^{\prime\prime}=-\frac{\varepsilon}{c^{2}}\beta qS^{\prime}\left(
u\right)  v^{\prime}.\\
\text{If }q^{\prime}  &  =u^{\prime}=v^{\prime}=0\text{\ then }q^{iv}%
=-\frac{\varepsilon}{c^{2}}\beta qS^{\prime}(u)w^{\prime}=-\frac{\varepsilon
}{c^{2}}\beta qS^{\prime}\left(  u\right)  v^{\prime\prime}=-\frac
{\varepsilon}{c}\beta qS^{\prime}\left(  u\right)  u^{\prime\prime\prime}.
\end{align*}

\end{proposition}

We use the fourth item in this list to prove

\begin{lemma}
\label{lem7}For any $\varepsilon>0$ and $c>0,$ if $p$ is a solution on
$\mathcal{U}_{\varepsilon,c}^{+}$ and $u^{\prime}\ge0$ on an interval
$(-\infty,\tau]$, then $q^{\prime} <0$ on $(-\infty,\tau)$.
\end{lemma}

\begin{proof}
If $u^{\prime}$ never changes sign, let $\sigma$ denote $\infty$. Otherwise,
suppose that $u^{\prime}$ first changes sign at $\sigma.$ If $\ q^{\prime
}\left(  \tau\right)  =0$ for some $\tau<\sigma$ and $\tau$ is the first zero
of $q^{\prime}$, then $q^{\prime\prime}\left(  \tau\right)  \geq0,$ and by the
fourth item of Proposition \ref{prop2}, $u^{\prime}\left(  \tau\right)  \leq
0$. From the definitions of $\sigma$ and $\tau,$ $u^{\prime}\left(
\tau\right)  =0$ and so $q^{\prime\prime}\left(  \tau\right)  =0.$ Since
$u^{\prime}$ does not change sign at $\tau,$ $u^{\prime\prime}\left(
\tau\right)  =0$ and so $q^{\prime\prime\prime}\left(  \tau\right)  =0.$ Hence
at $\tau$,
\[
u^{\prime}=u^{\prime\prime}=q^{\prime}=q^{\prime\prime}=q^{\prime\prime\prime
}=0.
\]
If $u^{\prime\prime\prime}(\tau)=0$ then $p\left(  \tau\right)  $ is an
equilibrium point, a contradiction. If $u^{\prime\prime\prime}\left(
\tau\right)  <0$ then $\tau$ is a local maximum of $u^{\prime},$ which is
inconsistent with the assumption that $u^{\prime}\geq0$ on $(-\infty,\sigma].$
\ Hence $u_{c}^{\prime\prime\prime}\left(  \tau\right)  >0.$ But then
$q^{iv}\left(  \tau\right)  <0.$ \ This again implies that $q^{\prime}>0$ on
some interval to the left of $\tau,$ contradicting the definition of $\tau.$
This completes the proof of Lemma \ref{lem7}.
\end{proof}

\begin{lemma}
\label{lem2c} If $p=\left(  u,v,w,q\right)  $ is a solution on $\mathcal{U}%
_{\varepsilon,c}^{+}$ then $w>0$ on an interval $(-\infty,\tau]$ with
$u\left(  \tau\right)  =u_{m}.$
\end{lemma}

\begin{proof}
Observe that $h\left(  u\right)  >q_{0}$ for $u_{0}<u<u_{m}$. It follows that
if $u_{0} < u <u_{m}$ and $q<q_{0}$ on an interval $(-\infty,t)$, then
$w^{\prime}>0$ on this interval. Hence $w>0$ as long as $u_{0}<u\le u_{m}$ and
$q<q_{0}$. (That is, if $u_{0}<u\le u_{m}$ and $q<q_{0}$ on $(-\infty,t]$,
then $w>0$ on this interval.) Since $u^{\prime}>0$ as long as $w=v^{\prime}
\ge0$, Lemma \ref{lem7} implies that $w>0$ as long as $u_{0}<u\le u_{m}$,
proving Lemma \ref{lem2c}.
\end{proof}

Hence the conditions $u\left(  0\right)  =u_{m}$ and $w>0$ on $(-\infty,0]$
determine a unique solution
\[
p_{\varepsilon,c}=\left(  u_{\varepsilon,c},v_{\varepsilon,c},w_{\varepsilon
,c},q_{\varepsilon,c}\right)
\]
on $\mathcal{U}_{\varepsilon,c}^{+}.$

Let
\[
\Omega=\left\{  \left(  u,v,w,q\right)  ~|~0<u<1,0<v<1,\frac{1}{1+\beta
}<q<1\right\}  .
\]
Since $\left(  u_{0},u_{0},0,q_{0}\right)  \in\Omega,$ it follows from
Proposition \ref{prop1} that if $\mathcal{U}_{\varepsilon,c}^{+}$ is a
homoclinic orbit, then it lies entirely in $\Omega.$

In the rest of this subsection let $c_{1}$ and $\varepsilon$ be chosen as in
Theorem \ref{thm1}.

\begin{lemma}
\label{lem3a} If $c\geq c_{1},$ then either $w_{\varepsilon,c}>0$ on $R,$ or
at the first zero of $w_{\varepsilon,c},$ $\ u_{\varepsilon,c}>u_{knee}.$
\end{lemma}

\begin{proof}
By the hypotheses on $c_{1}$ in Theorem \ref{thm1}, Lemma \ref{lem2} implies
that
\[
u_{0,c}\left(  T\left(  c\right)  \right)  >u_{knee.}%
\]

To extend this to $\varepsilon>0$ a comparison result is needed. Let
$p=\left(  u,v,w,q\right)  =p_{\varepsilon,c}.$ Lemma \ref{lem7} implies that
if $\varepsilon>0$ then $q<q_{0}$ on any interval $(-\infty,t]$ where $w>0,$
since in such an interval $v^{\prime}>0$ and $u^{\prime}>0.$ Also, as long as
$w>0$ we can consider $u$, $w$, and $q$ as functions of $v.$ Say that
$u=U\left(  v\right)  $, $w=W\left(  v\right)  $, and $q=Q(v)$. Then%
\begin{equation}
\left.
\begin{array}
[c]{c}%
U^{\prime}(v)=\frac{v-U(v)}{cW(v)}\\
W^{\prime}(v)=\frac{b^{2}(v-Q(v)S(U(v))}{W(v)}%
\end{array}
\right.  \label{a1}%
\end{equation}

We compare $w=W\left(  v\right)  $ with the solution when $\varepsilon=0.$ Let
$p_{1}=p_{0,c}.$ Then we can write $u_{1}=U_{1}\left(  v_{1}\right)  $,
$w_{1}=W_{1}\left(  v_{1}\right)  $, and $q=q_{0}$. The equations become%

\begin{equation}
\left.
\begin{array}
[c]{c}%
U_{1}^{\prime}(v)=\frac{v-U_{1}(v)}{cW_{1}(v)}\\
W_{1}^{\prime}(v)=\frac{b^{2}(v-q_{0}S(U_{1}(v))}{W_{1}(v)}%
\end{array}
\right.  \label{a2}%
\end{equation}
Since $\lambda_{1}\left(  c,\varepsilon\right)  >\lambda_{1}\left(
c,0\right)  $ (Lemma \ref{lem2b}), it is seen by considering eigenvectors of
the linearization of (\ref{1}) around $p_{0}$\footnote{The relevant matrix $B$
is given in Appendix B.} that for $v$ sufficiently close to $u_{0}$ (i.e. for
large negative $t$),%
\begin{equation}
\left\{
\begin{array}
[c]{c}%
U\left(  v\right)  <U_{1}\left(  v\right) \\
W\left(  v\right)  >W_{1}\left(  v\right)
\end{array}
\right.  . \label{a3}%
\end{equation}
If, at some first $\hat{v},$ one of these inequalities should fail while the
other still holds, then a contradiction results from comparing (\ref{a1}) and
(\ref{a2}), because $q<q_{0}$ and $S$ is increasing. For example, if
$U(\hat{v})=U_{1}(\hat{v})$ and $W(\hat{v})>W_{1}(\hat{v})>0$, then (\ref{a1})
and (\ref{a2}) imply that $U^{\prime}(\hat{v})<U_{1}^{\prime}(\hat{v})$, a
contradiction because $U<U_{1}$ on $(u_{0},\hat{v})$. Also, if $W(\hat
{v})=W_{1}(\hat{v})\geq0$ and $U(\hat{v})<U_{1}(\hat{v})$ then $W^{\prime
}(\hat{v})>W_{1}^{\prime}(\hat{v})$, since $q<q_{0}$ as long as $w\geq0$. This
is also a contradiction of the definition of $\hat{v}$.

If both inequalities fail at the same $\hat{v},$ then there is still a
contradiction because $q<q_{0}.$ Hence, if $W_{1}(v)\geq0$ for $u_{0}\leq
v\leq\hat{v}$ then (\ref{a3}) holds in this interval.

This implies that for any $c\in\left(  c_{1},c_{0}^{\ast}\right)  ,$ if
$w_{\varepsilon,c}$ has a first zero at $T\left(  \varepsilon,c\right)  ,$
then $v_{\varepsilon,c}\left(  T\left(  \varepsilon,c\right)  \right)
>v_{0,c}\left(  T\left(  0,c\right)  \right)  .$ In the proof of Lemma
\ref{lem2} it is shown that $v_{0,c}\left(  T\left(  0,c\right)  \right)
>v_{0,c_{1}}\left(  T\left(  0,c_{1}\right)  \right)  ,$ and combining these
shows that if $u_{\varepsilon,c}\left(  T\left(  \varepsilon,c\right)
\right)  \leq u_{knee}\ $and $v_{0,c_{1}}\left(  T\left(  0,c_{1}\right)
\right)  >q_{0}S\left(  u_{knee}\right)  $ then $w_{\varepsilon,c}^{\prime
}\left(  T\left(  \varepsilon,c\right)  \right)  >0.$ \ This contradiction
completes the proof of Lemma \ref{lem3a}.
\end{proof}

\begin{lemma}
\label{lem3}If $c>c_{0}^{\ast},$ then $v_{\varepsilon,c}>0$, $w_{\varepsilon
,c}>0,$ $u_{\varepsilon,c}^{\prime}>0,$ and $u_{\varepsilon,c}\rightarrow
\infty$ as $t\rightarrow\infty.$
\end{lemma}

\begin{proof}
This follows from Lemma \ref{lem2} and the comparison used to prove Lemma
\ref{lem3a}.
\end{proof}

\ \ We are now ready to apply a \textquotedblleft shooting\textquotedblright%
\ argument to obtain the fast pulse. Still with $\varepsilon$ as in Theorem
\ref{thm1}, for each $c>0$ let $p=p_{\varepsilon,c}$ and set%

\begin{align*}
&  \Lambda=\left\{  c\geq c_{1}~|~\text{There exist }t_{1},\text{ }%
t_{2}\text{, and }t_{3}\text{\ such that }0<t_{1}<t_{2}<t_{3}\text{ and
}\right. \\
&  \text{if }p=p_{\varepsilon,c}\text{ then }u^{\prime}>0\text{ on }%
[0,t_{1}),\text{ }u^{\prime}\left(  t_{1}\right)  =0,~u\left(  t_{2}\right)
=u_{0},\text{\ and either }u\left(  t_{3}\right)  =0\text{ or }q\left(
t_{3}\right)  =q_{0}.\\
&  \left.  \text{Further, }u^{\prime\prime}\left(  t_{1}\right)  <0,\text{
}u^{\prime}<0\text{ on }(t_{1},t_{2}]\text{ and }u<u_{0}\text{ on }%
(t_{2},t_{3}].\right\}
\end{align*}
\medskip

(See Figures \ref{fig2a}.)

\begin{center}
\begin{figure}[h]
\includegraphics[height=1.5 in, width =4  in]{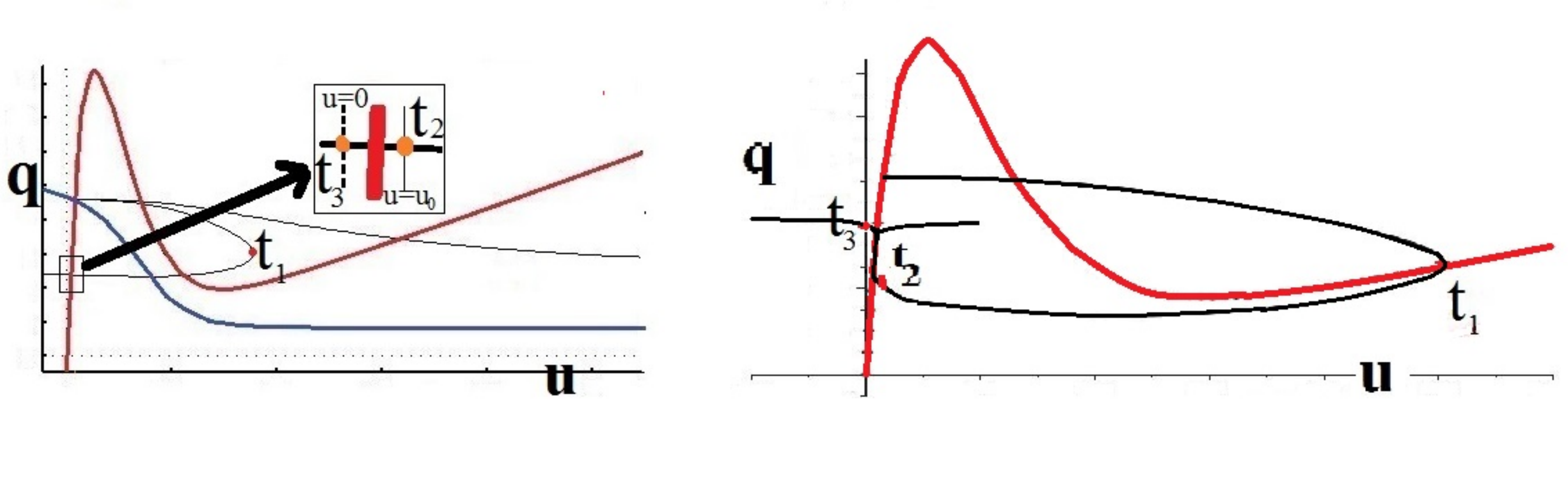}\caption{Each
figure shows one solution with $c\in\Lambda$ and one solution with
$c\protect\notin\Lambda$.}%
\label{fig2a}%
\end{figure}
\end{center}

\begin{lemma}
\label{lem5} $\Lambda$ is an open subset of the half line $c\geq c_{1}$.
\end{lemma}

\begin{proof}
Suppose that $c\in\Lambda,$ and choose $t_{3}=t_{3}\left(  c\right)  $ as in
the definition of $\Lambda.$ Note from (\ref{1}) that if $u_{c}\left(
t_{3}\right)  =0$ then there is a $\tau<t_{3}$ such that $v\left(
\tau\right)  =0,$ $v^{\prime}\left(  \tau\right)  \leq0.$ Also, $v^{\prime
\prime}<0$ if $v\leq0.$ Hence $v\left(  t_{3}\right)  <0$ and $u^{\prime
}\left(  t_{3}\right)  <0.$

Also, (\ref{1}) implies that if $q(t_{3})=q_{0}$ and $u(t_{3})<u_{0}$ then
$q^{\prime}(t_{3})>0$. Since $p_{c}(t)$ is a smooth function of $c$, uniformly
for $t$ in, say, $(-\infty,t_{3}\left(  c^{\ast}\right)  +1],$ it follows that
for $c$ in some neighborhood of $c^{\ast}$, $t_{i}\left(  c\right)  $ is
defined for $i=1,..3$ and all the inequalities in the definition of $\Lambda$
\ continue to hold, so that this neighborhood lies in $\Lambda.$ This proves
Lemma \ref{lem5}.
\end{proof}

The hypotheses of Theorem \ref{thm1} imply that $c_{1}\in\Lambda$, while by
Lemma \ref{lem3}, if $c>c_{0}^{\ast},$ then $c\notin\Lambda.$ The numbers
$t_{i}$ depend on $c,$ and when we need to emphasize this we will denote them
by $t_{i}\left(  c\right)  ,$ for $i=1,2,3$.

We now let $c^{\ast}=\sup\Lambda.$ (This is finite, by Lemma \ref{lem3}.)

\begin{lemma}
\label{lem4}With $\varepsilon$ as in Theorem \ref{thm1}, $\mathcal{U}%
_{\varepsilon,c^{\ast}}^{+}$ is a homoclinic orbit of (\ref{1}).
\end{lemma}

\begin{proof}
The proof depends on the fact that $c^{\ast}$is a boundary point of $\Lambda$
and lies in $\left(  c_{1},c_{0}^{\ast}\right)  .$ In Figure 4 we show several
graphs of $\left(  u,q\right)  $ which, if they occurred for $p_{\varepsilon
,c},$ would suggest (without quite implying) that $c$ was on the boundary of
$\Lambda.$ \ We must eliminate these and some other possibilities, and this
will imply that $p_{\varepsilon,c^{\ast}}\left(  \infty\right)  =\left(
u_{0},u_{0},0,q_{0}\right)  .$ The reader may want to review the definition of
$\Lambda$ before examining these figures.\medskip

\begin{figure}[h]
\includegraphics[height=1. in, width =5.5  in]{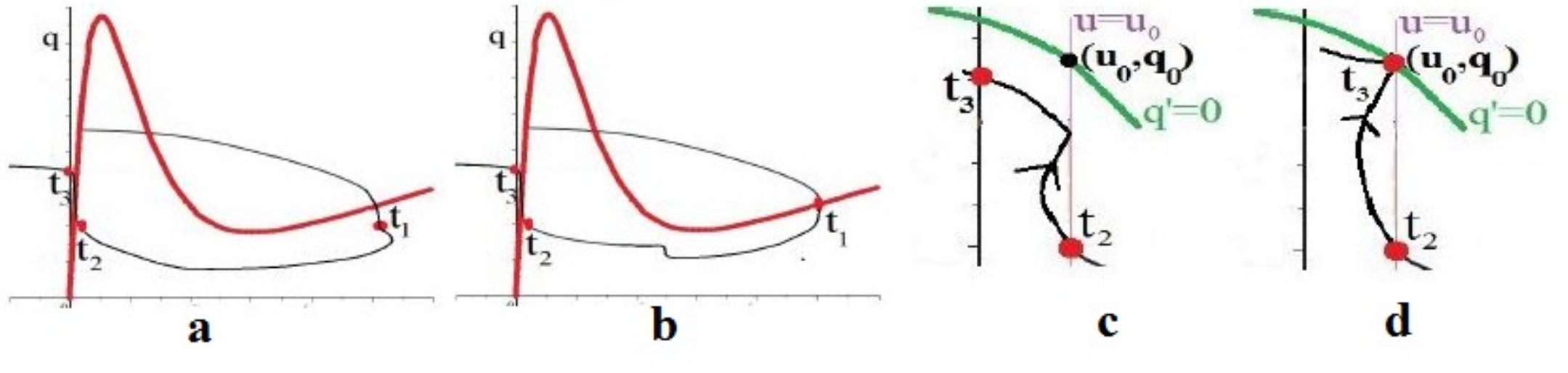}
\label{figcc}\caption{Several possibilities if $c$ is on the boundary of
$\Lambda$}%
\end{figure}

\medskip

We need additional lemmas. Recalling that $\varepsilon$ was chosen in the
statement of \ref{thm1}, we will now drop the $\varepsilon$-dependence of
$p_{\varepsilon,c}$ and its components from our notation, writing $p_{c}$ and
$\Lambda$. When the dependence of $p_{c}$ on $c$ is not crucial to an argument
we will use $u,v,w,q$ for its components.

\begin{lemma}
\label{lem5a}Suppose that $p=\left(  u,v,w,q\right)  $ is a non-constant
solution of (\ref{1}) satisfying one of the following sets of conditions at
some $\tau:$%
\begin{align*}
\text{(i)}~u^{\prime}\left(  \tau\right)   &  =0,\text{ }v^{\prime}\left(
\tau\right)  =0,~w^{\prime}\left(  \tau\right)  \leq0,\text{ }q^{\prime
}\left(  \tau\right)  \leq0,~u\left(  \tau\right)  \geq u_{0}\\
\text{(ii)}~u^{\prime}\left(  \tau\right)   &  =0,~v^{\prime}\left(
\tau\right)  <0,~w^{\prime}\left(  \tau\right)  =0\text{, }q^{\prime}\left(
\tau\right)  =0,~u\left(  \tau\right)  \geq u_{0}\\
\text{(iii)}~u^{\prime}\left(  \tau\right)   &  =0,~v^{\prime}\left(
\tau\right)  \leq0,~q^{\prime}\left(  \tau\right)  >0,~u\left(  \tau\right)
=u_{0}.\\
\text{(iv)~}u^{\prime}\left(  \tau\right)   &  =0,~v^{\prime}\left(
\tau\right)  \leq0,~w^{\prime}\left(  \tau\right)  >0,~q^{\prime}\left(
\tau\right)  \geq0,~u\left(  \tau\right)  \geq u_{0}%
\end{align*}
Then $p\left(  \tau\right)  \notin\mathcal{U}_{\varepsilon,c^{\ast}}^{+}.$
\end{lemma}

\begin{proof}
Suppose that (i) holds. Then $w^{\prime}\left(  \tau\right)  $ and $q^{\prime
}\left(  \tau\right)  $ cannot both vanish. If $w^{\prime}\left(  \tau\right)
=0$ and $q^{\prime}\left(  \tau\right)  <0$ then $w^{\prime\prime}\left(
\tau\right)  =-b^{2}q^{\prime}\left(  \tau\right)  S\left(  u\left(
\tau\right)  \right)  >0.$ Hence in some interval $(\tau-\delta,\tau)$,
\begin{equation}
w^{\prime}<0\text{ and }q^{\prime}<0. \label{6aa}%
\end{equation}
If $q^{\prime}\left(  \tau\right)  =0$ and $w^{\prime}\left(  \tau\right)  <0$
then
\begin{align*}
q^{\prime\prime}\left(  \tau\right)   &  =-\frac{\varepsilon}{c}\beta q\left(
\tau\right)  S^{\prime}\left(  u\left(  \tau\right)  \right)  u^{\prime
}\left(  \tau\right)  =0\\
q^{\prime\prime\prime}\left(  \tau\right)   &  =-\frac{\varepsilon}{c^{2}%
}\beta q\left(  \tau\right)  S^{\prime}\left(  u\left(  \tau\right)  \right)
v^{\prime}\left(  \tau\right)  =0\\
q^{\left(  iv\right)  }\left(  \tau\right)   &  =-\frac{\varepsilon}{c^{2}%
}\beta q\left(  \tau\right)  S^{\prime}\left(  u\left(  \tau\right)  \right)
w^{\prime}\left(  \tau\right)  >0.
\end{align*}
Once again we see that (\ref{6aa}) holds on some interval $(\tau-\delta,\tau)$.

Consider the \textquotedblleft backward\textquotedblright\ system satisfied by
$P\left(  s\right)  =p\left(  \tau-s\right)  .$ If $P=\left(  U,V,W,Q\right)
$ then
\begin{equation}
\left.
\begin{array}
[c]{c}%
U^{\prime}=\frac{U-V}{c}\\
V^{\prime}=-W\\
W^{\prime}=b^{2}\left(  QS\left(  U\right)  -V\right) \\
Q^{\prime}=\frac{\varepsilon}{c}\left(  Q+\beta QS\left(  U\right)  -1\right)
\end{array}
\right.  . \label{6}%
\end{equation}
Also,
\begin{equation}
U^{\prime}\left(  0\right)  =0,V^{\prime}\left(  0\right)  =0,~W^{\prime
}\left(  0\right)  \geq0,\text{ and }Q^{\prime}\left(  0\right)  \geq0.
\label{6a}%
\end{equation}

From (\ref{6aa}) and (\ref{6a}) it follows that on some interval $0<s<\delta
$,
\begin{equation}
U^{\prime}>0,V^{\prime}<0,W^{\prime}>0\text{ and }Q^{\prime}>0. \label{7}%
\end{equation}
We claim that these inequalities hold for all $s>0.$ If, on the contrary, one
of them fails at a first $s_{0}>0,$ then
\begin{equation}
U\left(  s_{0}\right)  >U\left(  0\right)  ,V\left(  s_{0}\right)  <V\left(
0\right)  ,W\left(  s_{0}\right)  >W\left(  0\right)  ,\text{ and }Q\left(
s_{0}\right)  >Q\left(  0\right)  . \label{8}%
\end{equation}
But (\ref{8}), (\ref{6}), and (\ref{6a}) imply that at $s_{0},$ all of the
inequalities in (\ref{7}) still hold, because $S^{\prime}>0.$ This
contradiction implies that $U$, $W$, and $Q$ continue to increase, and $V$
continues decrease on $0<s<\infty,$ and in particular, $U$ does not tend to
$u_{0}$ as $s\rightarrow\infty.$ \ Thus, $p\left(  \tau\right)  \notin%
\mathcal{U}_{\varepsilon,c^{\ast}}^{+}.$

The proofs in cases (ii), (iii) and (iv) are similar and left to the reader.
\ This completes the proof of Lemma \ref{lem5a}.
\end{proof}

We now begin our study of the properties of $p_{c\ast}.$ Lemma \ref{lem5a}
will assist us in proving the following result.

\begin{lemma}
\label{lem6}The number $t_{1}\left(  c^{\ast}\right)  $ is still defined, as
the first zero of $u_{c^{\ast}}^{\prime},$ and $u_{c^{\ast}}^{\prime\prime
}\left(  t_{1}\left(  c^{\ast}\right)  \right)  <0.$ Either $\mathcal{U}%
_{\varepsilon,c^{\ast}}^{+}$ is homoclinic or $t_{2}\left(  c^{\ast}\right)  $
is still defined, as the first zero of $u-u_{0}$. Also, if $\mathcal{U}%
_{\varepsilon,c^{\ast}}^{+}$ is not homoclinic then $u_{c^{\ast}}^{\prime}<0$
on $(t_{1},t_{2}].$
\end{lemma}

\begin{remark}
This lemma eliminates the graph in Figure \ref{figcc}-a.
\end{remark}

\begin{proof}
Suppose that $t_{1}\left(  c^{\ast}\right)  $ is not defined. Then
$u_{c^{\ast}}^{\prime}>0$ on $\left(  -\infty,\infty\right)  .$ \ Since
$p_{0}$ is the only equilibrium point of (\ref{1}), this implies that for some
$\tau$, $v_{c^{\ast}\left(  \tau\right)  }>u_{c^{\ast}}\left(  \tau\right)
>1$ and $v_{c^{\ast}}^{\prime}\left(  \tau\right)  >0.$ Then these
inequalities hold at $\tau$ for nearby $c,$ and by Proposition \ref{prop1},
$v_{c}\left(  t\right)  >1$ for $t>\tau.$ Hence $u_{c}>1$ on $[\tau,\infty)$
and so $c\notin\Lambda,$ contradicting the definition of $c^{\ast}.$ Therefore
$t_{1}\left(  c^{\ast}\right)  $ is defined.

We now show that $u_{c^{\ast}}^{\prime\prime}\left(  t_{1}\right)  <0.$ Again
assume that $p=p_{c^{\ast}},$ and suppose that $u^{\prime\prime}\left(
t_{1}\right)  =0.$ If $u^{\prime\prime\prime}\left(  t_{1}\right)  <0$ then
$t_{1}$ is a local maximum of $u^{\prime},$ which is not possible because
$t_{1}$ is the first zero of $u^{\prime}.$ \ Hence at $t_{1}$, $u^{\prime
\prime\prime}=\frac{b^{2}}{c^{\ast}}\left(  v-qS\left(  u\right)  \right)
\geq0.$ If $u^{\prime\prime\prime}\left(  t_{1}\right)  =0,$ then at $t_{1},$
\[
u^{\left(  iv\right)  }=\frac{w^{\prime\prime}}{c}=-\frac{b^{2}}{c}q^{\prime
}S\left(  u\right)  >0,
\]
by Lemma \ref{lem7}. This implies that $u^{\prime}$ changes sign from negative
to positive at $t_{1},$ again a contradiction of the definition of $t_{1}.$
Hence at $t_{1},$
\[
u^{\prime\prime\prime}=\frac{b^{2}}{c^{\ast}}\left(  v-qS\left(  u\right)
\right)  >0,
\]
or $q\left(  t_{1}\right)  <\frac{v\left(  t_{1}\right)  }{S\left(  u\left(
t_{1}\right)  \right)  }=\frac{u\left(  t_{1}\right)  }{S\left(  u\left(
t_{1}\right)  \right)  }$, since $u^{\prime}\left(  t_{1}\right)  =0.$ Also,
$u^{\prime\prime}>0$ in some interval $\left(  t_{1},t_{1}+\delta\right)  .$
However $u$ is bounded by $1$ and does not tend to a limit above $u_{0}.$
Therefore $u^{\prime}$ changes sign at some $\tau>t_{1}\left(  c^{\ast
}\right)  $. Since $u^{\prime}\geq0$ on $(-\infty,\tau],$ $v\geq u$ on this
interval. At $\tau$, $u^{\prime\prime}\leq0,$ and so there is a point $\sigma$
in $\left(  t_{1},\tau\right)  $ such that $u^{\prime\prime}=0$ and
$u^{\prime\prime\prime}=\frac{v^{\prime\prime}-u^{\prime\prime}}{c}%
=\frac{w^{\prime}}{c}\leq0.$ Hence at $\sigma,$ $q\geq\frac{v}{S\left(
u\right)  }\geq\frac{u}{S\left(  u\right)  }.$ \ But by Lemma \ref{lem2a}
$u\left(  t_{1}\right)  \geq u_{knee},$ and $h\left(  u\right)  =\frac
{u}{S\left(  u\right)  }$ is increasing in $\left(  u_{knee,},\infty\right)
,$ so $q\left(  \sigma\right)  >q\left(  t_{1}\right)  ,$ again a
contradiction of Lemma \ref{lem7}. We have therefore proved the first sentence
of Lemma \ref{lem6}.

\begin{lemma}
\label{lem8}If $p=p_{c^{\ast},}$ then $u^{\prime}<0$ as long after $t_{1}$ as
$q^{\prime}\leq0.$
\end{lemma}

\begin{proof}
Suppose instead that there is a first $\tau>t_{1}$ such that $q^{\prime}\leq0$
on $(-\infty,\tau]$ but $u^{\prime}\left(  \tau\right)  =0$. Then
$u^{\prime\prime}\left(  \tau\right)  \geq0.$ First consider the case
$q^{\prime}<0$ on $(-\infty,\tau].$

If $u^{\prime\prime}\left(  \tau\right)  >0$ then $u_{c^{\ast}}^{\prime}$
changes from negative to positive before $q^{\prime}=0,$ and this will be true
as well for $c$ close to $c^{\ast},$ contradicting the definition of $c^{\ast
}.$ Hence suppose that $u^{\prime\prime}\left(  \tau\right)  =\frac{v^{\prime
}\left(  \tau\right)  }{c^{\ast}}=0.$ If $u^{\prime\prime\prime}>0,$ then
$u^{\prime}$ has a local minimum at $\tau,$ contradicting the definition of
$\tau.$ \ Hence%
\[
u^{\prime\prime\prime}\left(  \tau\right)  =\frac{w^{\prime}\left(
\tau\right)  }{c^{\ast}}\leq0.
\]
But now the conditions in (i) of Lemma \ref{lem5a} are satisfied, giving a contradiction.

We have left to consider the case that $q^{\prime}\left(  \tau\right)
=u^{\prime}\left(  \tau\right)  =0.$ Then $q^{\prime\prime}\left(
\tau\right)  =0.$ If $u^{\prime\prime}\left(  \tau\right)  >0$ then
$q^{\prime\prime\prime}\left(  \tau\right)  <0$ so $q^{\prime}<0$ in an
interval $\left(  \tau,\tau+\delta\right)  .$ Hence in this case, $u^{\prime}$
changes sign (from negative to positive) before $q^{\prime}>0.$ For $c$ close
to $c^{\ast}$ there are two possibilities: either $u_{c}^{\prime}$ changes
sign from negative to positive before $q_{c}^{\prime}>0,$ and so before
$u=u_{0}$, or else $u_{c}^{\prime}<0$ in a neighborhood of $\tau,$ but in a
neighborhood of, say, $\tau+\frac{1}{2}\delta$, $u_{c}^{\prime}>0$ and
$u>u_{0}$. (See Figure \ref{figure3}) Neither of these possibilities occurs if
$c\in\Lambda,$ so once again, $c^{\ast}\notin\partial\Lambda,$ a
contradiction. This proves Lemma \ref{lem8}.
\end{proof}

\begin{figure}[h]
\includegraphics[height=2 in, width =2  in]{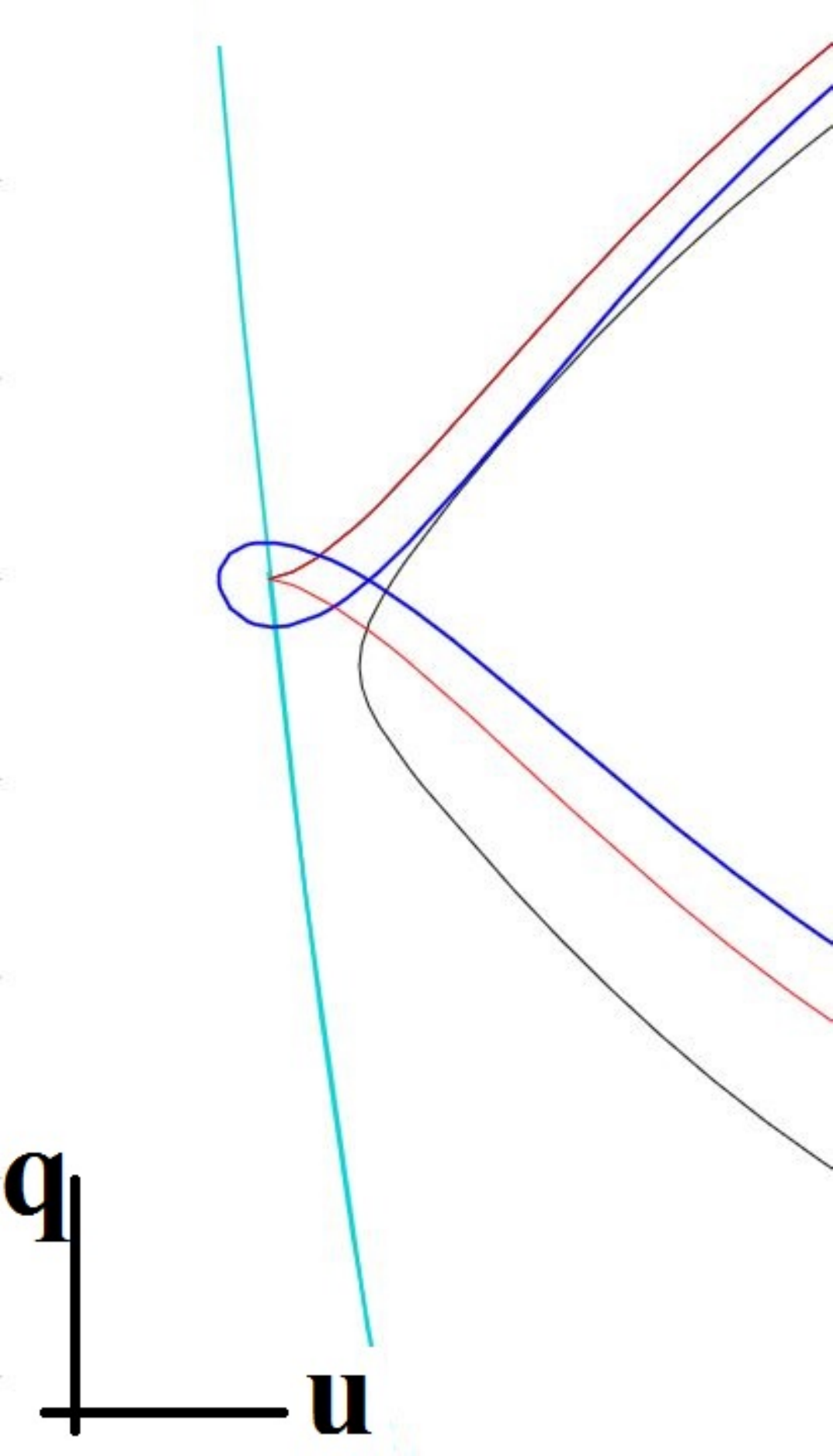}\caption{$(u_{c}%
,q_{c})$ for $c=c^{\ast}$ (with cusp), and two other solutions with
$c\protect\notin\Lambda$. The graph of a decreasing function is part of the
$q$ nullcline.}%
\label{figure3}%
\end{figure}

It follows that there is a first $\tau_{1}>t_{1}$ such that $q_{c^{\ast}%
}^{\prime}\left(  \tau_{1}\right)  =0.$ Also, $u_{c^{\ast}}^{\prime}\left(
\tau_{1}\right)  <0,$ and (equivalently by the fourth item of Proposition
\ref{prop2}) $q_{c^{\ast}}^{\prime\prime}\left(  \tau_{1}\right)  >0.$

\begin{lemma}
\label{lem9}$q_{c^{\ast}}^{\prime}>0$ and $u_{c^{\ast}}^{\prime}<0$ as long
after $\tau_{1}$ as $u_{c^{\ast}}\geq u_{0}.$
\end{lemma}

\begin{proof}
This lemma eliminates the graph in Figure \ref{figcc}-b.

Let $p=p_{c^{\ast}}.$ Since $q^{\prime\prime}\left(  \tau_{1}\right)  >0,$
$q^{\prime}>0$ and $u^{\prime}<0$ on some interval $(\tau_{1},\tau_{1}%
+\delta]$ with $\delta>0.$ We claim that $q^{\prime}>0$ on any such
half-closed interval in which $u^{\prime}<0.$ This follows because, by
Proposition \ref{prop2}, $q^{\prime\prime}>0$ at any point where $q^{\prime
}=0$ and $u^{\prime}<0.$

We next show that $u^{\prime}<0$ on any interval $(\tau_{1},\tau_{1}+\delta]$
in which $q^{\prime}>0$ and $u\geq u_{0}.$ If not, then there is a first
$\sigma>\tau_{1}$ with $u^{\prime}\left(  \sigma\right)  =0,$ $q^{\prime
}\left(  \sigma\right)  >0$ and $u\geq u_{0}$ on $(-\infty,\sigma].$ Then
$u^{\prime\prime}\left(  \sigma\right)  \geq0.$ \ If $u^{\prime\prime}\left(
\sigma\right)  >0,$ then $u_{c^{\ast}}^{\prime}>0$ in some interval $\left(
\sigma,\sigma+\delta\right)  .$ In this case, for $c$ close enough to
$c^{\ast},$ $u_{c}^{\prime}$ changes sign after $t_{1}$ but before
$u_{c}<u_{0}$ or else $u_{c}$ crosses $u_{0}$ and back again, and such $c$
cannot lie in $\Lambda,$ a contradiction.

Hence, $u^{\prime\prime}\left(  \sigma\right)  =0.$ \ But then, because
$u\left(  \sigma\right)  =v\left(  \sigma\right)  ,$
\[
u^{\prime\prime\prime}\left(  \sigma\right)  =\frac{b^{2}}{c^{\ast}}\left(
v-qS\left(  u\right)  \right)  =\frac{b^{2}}{c^{\ast}}\left(  u-qS\left(
u\right)  \right)  .
\]
In the region where $q^{\prime}>0$ and $u\geq u_{0},$ $q<\frac{u}{s\left(
u\right)  }.$ Hence, $u^{\prime\prime\prime}\left(  \sigma\right)  >0$ and
again $u_{c^{\ast}}^{\prime}>0$ in an interval to the right of $\sigma$ but
before $u_{c^{\ast}}<u_{0},$ a contradiction as before.

The only other possibility contradicting Lemma \ref{lem9} is that there is a
first $\tau>\tau_{1}\left(  c^{\ast}\right)  $ where $u_{c^{\ast}}\left(
\tau\right)  \geq u_{0}$ and $q_{c^{\ast}}^{\prime}\left(  \tau\right)
=u_{c^{\ast}}^{\prime}\left(  \tau\right)  =0.$ We consider two cases: (a)
\ $q_{c^{\ast}}\left(  \tau\right)  <q_{0}$ and $u_{c^{\ast}}\left(
\tau\right)  >u_{0},$ and (b) $q_{c^{\ast}}\left(  \tau\right)  =q_{0}%
,~u_{c^{\ast}}\left(  \tau\right)  =u_{0}.$ First consider (a). In an interval
$\left(  \tau-\delta,\tau\right)  ,$ $q_{c^{\ast}}^{\prime}>0$, $u_{c^{\ast}%
}^{\prime}<0,$ and $u_{c^{\ast}}>u_{0},$ and so at $\tau,$ if $p=p_{c^{\ast}%
},$ then
\[
u^{\prime}=0,~u^{\prime\prime}\geq0,~q^{\prime}=0,~q^{\prime\prime}=0.
\]
Also,
\[
q^{\prime\prime\prime}=-\frac{\varepsilon}{c}\beta qS^{\prime}\left(
u\right)  u^{\prime\prime}\leq0.
\]
But $u_{c^{\ast}}^{\prime\prime}(\tau)>0$ is impossible because it means that
even for nearby $c,$ $u_{c}^{\prime}>0$ after $t_{1}$ but before $u=u_{0}.$
Therefore at $\tau$, $q^{\prime\prime\prime}=0$ and $u^{\prime\prime}=0.$
Then
\[
q^{\left(  iv\right)  }=-\frac{\varepsilon}{c}qS^{\prime}\left(  u\right)
u^{\prime\prime\prime}.
\]
But on the nullcline $q^{\prime}=0,$ with $u>u_{0}$, $u^{\prime}=0,$ and
$u^{\prime\prime}=0,$
\begin{equation}
u^{\prime\prime\prime}=\frac{v^{\prime\prime}}{c^{\ast}}=\frac{b^{2}}{c^{\ast
}}\left(  v-qS\left(  u\right)  \right)  =\frac{b^{2}}{c^{\ast}}\left(
u-qS\left(  u\right)  \right)  >0. \label{4}%
\end{equation}
This implies that $u^{\prime}$ has a local minimum at $\tau$, whereas we know
that $u^{\prime}<0$ in $\left(  t_{1},\tau\right)  .$ This contradicts the
definition of $\tau.$

Turning to case (b), we now have that at $\tau,$
\[
u^{\prime}=0,q^{\prime}=0,q^{\prime\prime}=0,\left(  u,q\right)  =\left(
u_{0},q_{0}\right)  .
\]

Thus $w^{\prime}\left(  \tau\right)  =0.$ If $u^{\prime\prime}\left(
\tau\right)  >0$ then $u^{\prime}>0$ to the right of $\tau.$ As before, if $c$
is close to $c^{\ast}$ then either $u_{c}$ crosses $u_{0}$ twice, or $p_{c}$
doesn't reach the region $u<u_{0}$ before $u^{\prime}>0,$ both of which mean
that $c\notin\Lambda$.

If $u^{\prime\prime}\left(  \tau\right)  =\frac{w\left(  \tau\right)  }{c}=0$
then $p\left(  \tau\right)  $ is again an equilibrium point. The third
possibility, $u^{\prime\prime}\left(  \tau\right)  =\frac{v^{\prime}\left(
\tau\right)  }{c^{\ast}}<0$ implies that (ii) of Lemma \ref{lem5a} is
satisfied, and thus again gives a contradiction. This completes the proof of
Lemma \ref{lem9}.
\end{proof}

If $u>u_{0}$ on $R$ then $u_{c^{\ast}}^{\prime}<0$ and $q_{{\ast}}^{\prime}>0$
on $(\tau_{1},\infty)$, and $\mathcal{U}_{\varepsilon,c^{\ast}}^{+}$ is
homoclinic. This proves Lemma \ref{lem6}.
\end{proof}

Thus, for $p=p_{c^{\ast}}$ if $\mathcal{U}_{\varepsilon,c^{\ast}}^{+}$ is not
homoclinic then $t_{2}$ exists with $u\left(  t_{2}\right)  =u_{0}$ and
$u^{\prime}<0$ on $(t_{1},t_{2}].$ However, there is no $t_{3}$ such that
$u<u_{0}$ on $(t_{2},t_{3}]$ and either $u\left(  t_{3}\right)  =0$ or
$q(t_{3})=q_{0}$, for otherwise $c^{\ast}\in\Lambda,$ and this has already
been ruled out.

Therefore if $t>t_{2}\left(  c^{\ast}\right)  $ then $0<u_{c^{\ast}}\leq
u_{0}$ and $q_{c^{\ast}}\leq q_{0},$ for otherwise nearby values of $c$ are
once again not in $\Lambda.$ If $\mathcal{U}_{\varepsilon,c^{\ast}}^{+}$ is
not homoclinic and $p=p_{c^{\ast}}$, then there must be a first $\tau
>t_{2}\left(  c^{\ast}\right)  $ with $u\left(  \tau\right)  =u_{0}$,
$u^{\prime}\left(  \tau\right)  =0,$ $q\left(  \tau\right)  \leq q_{0}$ and
$u^{\prime\prime}\left(  \tau\right)  \leq0.$

Suppose that this is the case and also $q\left(  \tau\right)  <q_{0}.$ (This
is pictured in Figure \ref{figcc}-c.)

Then $q^{\prime}\left(  \tau\right)  >0.$ If $u^{\prime\prime}\left(
\tau\right)  \leq0,$ then (iii) of Lemma (\ref{lem5a}) applies and gives a
contradiction. Hence $q(\tau)=q_{0}.$ Then at $\tau,$ $q^{\prime}=u^{\prime
}=w^{\prime}=0$ and $u^{\prime\prime}=\frac{v^{\prime}}{c^{\ast}}<0.$ (This is
pictured in Figure \ref{figcc}-d. \ But this is case (ii) of Lemma \ref{lem5a}
and so also impossible. \ 

We have established that if $\mathcal{U}_{\varepsilon,c^{\ast}}^{+}$ is not
homoclinic (with $u>u_{0}$ on $R$ ) then for large $t,$ $0<u_{c^{\ast}}\left(
t\right)  <u_{0}$ and $q_{c^{\ast}}^{\prime}>0.$ This is only possible if
$\mathcal{U}_{\varepsilon,c^{\ast}}^{+}$ is homoclinic (with $q_{c^{\ast}%
}<q_{0}$ and $u_{c^{\ast}}^{\prime}>0$ for large $t$). This proves Lemma
\ref{lem4}.
\end{proof}

To complete the proof of Theorem \ref{thm1} we look for a second homoclinic
orbit, with $c<c_{1}.$

\subsection{The slow pulse}

\ Again we adapt the method in \cite{hmc}. It is stated so as to be useful in
the proofs of Theorems \ref{thm2} and \ref{thm3}, as well as Theorem
\ref{thm1}.

\begin{lemma}
\label{lem10}There are $\hat{c}>0$ and $M>0,$ both independent of
$\varepsilon,$ such that if $0<c<\tilde{c}$ and $\frac{\varepsilon}{c}>M$ then
the solution $p_{c}$ remains in the region $v>u$ on $\left(  -\infty
,\infty\right)  ,$ and $u$ crosses $u=1.$
\end{lemma}

\begin{proof}
Since the proof uses some of the easier parts of the proof of Lemma
\ref{lem2}, it is included in the appendix.
\end{proof}

From here to the end of this section the parameters $\varepsilon$ and $c_{1}$
remain as in the previous subsection. Lemma \ref{lem10} implies that if we
extend $\Lambda$ to $\left(  0,c_{0}^{\ast}\right)  ,$ with otherwise the same
definition as above, then $\inf\Lambda>0.$ This suggests that $\inf\Lambda$
corresponds to a homoclinic orbit. The problem with this argument is that the
concept of \textquotedblleft front\textquotedblright\ in the sense used in the
method of geometric perturbation, breaks down for small $c.$ The slow
homoclinic orbit is not close, even up to the first zero of $u_{c}^{\prime},$
to the front found when $\varepsilon=0.$ More precisely, our proof of the
first sentence of Lemma \ref{lem6}\ is no longer valid, because we cannot
assert that the first zero of $u_{c}^{\prime}$ occurs with $u_{c}>u_{knee}.$
Hence we must modify our \textquotedblleft shooting set\textquotedblright\ on
the $c$ axis. This requires several steps.

Our argument from here no longer refers to a point $t_{1}$ where
$u_{c}^{\prime}$ changes sign, but instead considers solutions such that
$q_{c}^{\prime}$ changes sign. Let
\begin{align*}
\Sigma &  =\left\{  c\in(0,c_{1}]~|~\text{There is a }\tau_{1}>0\text{ such
that }q_{c}^{\prime}<0\text{ on }\left(  -\infty,\tau_{1}\right)  ,\right. \\
&  \left.  \text{ }q_{c}^{\prime}\left(  \tau_{1}\right)  =0\text{, and }%
u_{c}^{\prime}\left(  \tau_{1}\right)  <0\right\}  .
\end{align*}
Our argument does not require that if $c\in\Sigma$ then $u_{c}^{\prime}$ has
only one zero in $\left(  -\infty,\tau_{1}\right)  ,$ though numerically this
appears to be the case.
\[
\Sigma_{1}=\left\{  c\in\Sigma~|~q_{c}^{\prime}>0\text{ on any interval
}\left(  \tau_{1},T\right)  \text{ in which }u_{c}>0\text{ and }q_{c}%
<q_{0}\right\}  .
\]
Recall that $\varepsilon$ and $c_{1}$ were chosen so that $u_{c_{1}}^{\prime}$
has a unique zero. As in the proof of Lemma \ref{lem9}, this implies that
$q_{c_{1}}^{\prime}$ has a unique zero, say $\tau_{c_{1}},$ and $u_{c_{1}%
}^{\prime}\left(  \tau_{c_{1}}\right)  <0.$ Hence $c_{1}\in\Sigma_{1}$. Also,
Lemma \ref{lem10} shows that there is an interval $\left(  0,c_{2}\right)  $
which contains no points of $\Sigma.$

Let
\[
c_{3}=\sup\left\{  c<c_{1}~|~c\notin\Sigma\right\}  .
\]

\begin{lemma}
\label{lem11}There is a $\tau_{1}$ such that $q_{c_{3}}^{\prime}<0$ on
$\left(  -\infty,\tau_{1}\right)  ,$ $q_{c_{3}}^{\prime}\left(  \tau
_{1}\right)  =0,$ $q_{c_{3}}^{\prime\prime}\left(  \tau_{1}\right)  =0,$ and
$q_{c_{3}}^{\prime\prime\prime}\left(  \tau_{1}\right)  <0.$
\end{lemma}

\begin{proof}
If $q_{c_{3}}^{\prime}<0$ on $R$ then there is a $\sigma>0$ such that
$u_{c_{3}}\left(  \sigma\right)  =1$ and $u_{c_{3}}^{\prime}\left(
\sigma\right)  >0.$ From the continuity of $p_{c}\left(  t\right)  $ with
respect to $c,$ the same is true for $u_{c}$ if $c$ is sufficiently close to
$c_{3}.$ \ In particular, again $q_{c}^{\prime}<0$ on $\left(  -\infty
,\infty\right)  $. But then $c\notin\Sigma,$ contradicting the definition of
$c_{3}$.

Therefore a first $\tau_{1}$ is defined such that $q_{c_{3}}^{\prime}\left(
\tau_{1}\right)  =0$. Then $q_{c_{3}}^{\prime\prime}\left(  \tau_{1}\right)
\geq0.$ Also, by Proposition \ref{prop2}, $q_{c_{3}}^{\prime\prime}\left(
\tau_{1}\right)  =-\beta S^{\prime}\left(  u_{c_{3}}\left(  \tau_{1}\right)
\right)  u_{c_{3}}^{\prime}\left(  \tau_{1}\right)  .$ If $q_{c_{3}}%
^{\prime\prime}\left(  \tau_{1}\right)  >0$ then by the implicit function
theorem, $\tau_{1}\left(  c\right)  $ is defined for nearby $c$ as the first
zero of $q_{c}^{\prime},$ with $q_{c}^{\prime\prime}\left(  \tau_{1}\left(
c\right)  \right)  >0$ and $q_{c}^{\prime}<0$ on $\left(  -\infty,\tau
_{1}\left(  c\right)  \right)  ,$ contradicting the definition of $c_{3}.$
Hence $q_{c_{3}}^{\prime\prime}\left(  \tau_{1}\right)  =u_{c_{3}}^{\prime
}\left(  \tau_{1}\right)  =0$. If $q_{c_{3}}^{\prime\prime\prime}\left(
\tau_{1}\right)  >0$ then $\tau_{1}$ is a local minimum of $q_{c_{3}}^{\prime
},$ contradicting the definition of $\tau_{1}.$ If $q_{c_{3}}^{\prime
\prime\prime}\left(  \tau_{1}\right)  =0$ then $q_{c_{3}}^{iv}\left(  \tau
_{1}\right)  =-\frac{\varepsilon}{c_{3}}\beta q_{c_{3}}\left(  \tau
_{1}\right)  S^{\prime}\left(  u_{c_{3}}\left(  \tau_{1}\right)  \right)
w_{c_{3}}^{\prime}\left(  \tau_{1}\right)  ,$ and since $q_{c_{3}}^{\prime
}\left(  \tau_{1}\right)  =0$ and $q_{c_{3}}\left(  \tau_{1}\right)  <q_{0},$
$w_{c_{3}}^{\prime}(\tau_{1})>0$ and $q_{c_{3}}^{iv}\left(  \tau_{1}\right)
<0.$ This implies that $q_{c_{3}}^{\prime}>0$ on an interval $\left(  \tau
_{1}-\delta,\tau_{1}\right)  ,$ again a contradiction. Hence $q_{c_{3}%
}^{\prime\prime\prime}\left(  \tau_{1}\right)  <0,$ completing the proof of
Lemma \ref{lem11}.
\end{proof}

Thus, $q_{c_{3}}^{\prime}<0$ in some interval $\left(  \tau_{1},\tau
_{1}+\delta\right)  .$ This result implies that $c_{3}\notin\Sigma.$ However
the interval $(c_{3},c_{1}]\subset\Sigma.$ Lemma \ref{lem11} also implies that
points in $(c_{3},c_{1}]$ near to $c_{3}$ are not in $\Sigma_{1},$ since the
corresponding solutions on $\mathcal{U}_{\varepsilon,c}^{+}$ must have a
change of sign of $q_{c}^{\prime}$ from positive to negative after $\tau
_{1}\left(  c\right)  .$ Let
\[
c_{\ast}=\inf\left\{  c>c_{3}~|~c\in\Sigma_{1}\right\}  .
\]
We claim that $\mathcal{U}_{c_{\ast}}^{+}$ is a homoclinic orbit.

The proof uses techniques very similar to those above. First observe that
$c_{\ast}>c_{3}$ and $c_{\ast}\in\Sigma.$ Therefore $\tau_{1}=\tau_{1}\left(
c_{\ast}\right)  $ is defined as in the definition of $\Sigma.$ Then use the
following result.

\begin{lemma}
\label{lemslow2}If $c\in\Sigma_{1},$ then $u_{c}^{\prime}<0$ on any interval
$\left[  \tau_{1}\left(  c\right)  ,\tau_{1}\left(  c\right)  +\delta\right]
$ in which $u_{c}\geq u_{0}.$
\end{lemma}

\begin{proof}
If $u_{c}^{\prime}=0$ at some first $\sigma>\tau_{1}$ with $u_{c}\left(
\sigma\right)  \geq u_{0},$ then $u_{c}^{\prime\prime}\left(  \sigma\right)
\geq0.$ But in the region where $u\geq u_{0}$ and $q^{\prime}>0,$ $w^{\prime}$
is positive, and this implies that $p_{c}$ crosses into $q^{\prime}<0,$ a
contradiction of the definition of $\Sigma_{1}.$
\end{proof}

\begin{corollary}
If $\mathcal{U}_{\varepsilon,c_{\ast}}^{+}$ is not homoclinic then there is a
$t_{2}>\tau_{1}$ such that $u_{c_{\ast}}\left(  t_{2}\right)  =u_{0}$ and
$u_{c_{\ast}}^{\prime}<0$ on $\left[  \tau_{1},t_{2}\right]  .$ Further,
$u_{c_{\ast}}<u_{0}$ on $\left(  t_{2},\infty\right)  $.
\end{corollary}

\begin{proof}
Let $p=p_{c_{\ast}}.$ Lemma \ref{lemslow2} implies the existence of $t_{2}.$
Suppose there is a first $\sigma>t_{2}$ with $u\left(  \sigma\right)  =u_{0}.$
From the definitions of $\Sigma_{1}$ and $c_{\ast},$ $q^{\prime}\geq0$ on
$[\sigma,\infty).$ Since $\frac{1}{1+\beta S\left(  u\right)  }$ is decreasing
and $q^{\prime}>0$ if $q<\inf\frac{1}{1+\beta S},$ $u$ cannot increase
indefinitely. Hence there is a $\rho\geq\sigma$ with
\begin{align*}
u\left(  \rho\right)   &  \geq u_{0},u^{\prime}\left(  \rho\right)
=0,u^{\prime\prime}\left(  \rho\right)  =\frac{v^{\prime}\left(  \rho\right)
}{c}=\frac{w\left(  \rho\right)  }{c}\leq0\\
~w^{\prime}\left(  \rho\right)   &  >0,~q^{\prime}\left(  \rho\right)  \geq0.
\end{align*}
A contradiction\ then\ results\ from\ (iv) of Lemma \ref{lem5a}.
\end{proof}

Now apply the technique of Lemma \ref{lem5}, including use of Proposition
\ref{prop2} and Lemma \ref{lem5a}, to show that $u_{c_{\ast}}>0$ and
$q_{c_{\ast}}<q_{0}$ on $(t_{2},\infty)$. In particular, Lemma \ref{lem5a} is
used to show that there is no $t>t_{2}$ (in fact, no $t$ at all) with $\left(
u_{c_{\ast}}\left(  t\right)  ,q_{c_{\ast}}\left(  t\right)  \right)  =\left(
u_{0},q_{0}\right)  .$ It follows that on $\left(  t_{2},\infty\right)  ,$
$q_{c^{\ast}}^{\prime}>0,$ and so indeed, $\mathcal{U}_{\varepsilon,c_{\ast}%
}^{+}$ is homoclinic. This completes the proof of Theorem \ref{thm1}.

\subsection{Proofs of Theorems \ref{thm1a}, \ref{thm2} and \ref{thm3}.}

As mentioned above, Theorem \ref{thm1a} follows from Theorems \ref{thm1} and
\ref{thm2}. In Theorem \ref{thm2} $\varepsilon$ is not fixed. Also, $c_{1}$ is
any number in $\left(  0,c_{0}^{\ast}\right)  ,$ which however is fixed at
this stage for the rest of this section.

Let $I=\left[  c_{1},c_{0}^{\ast}+1\right]  $. We note that the unstable
manifold $\mathcal{U}_{\varepsilon,c}^{+}$ varies continuously with $\left(
\varepsilon,c\right)  $ for $\varepsilon\geq0$ and $c\in I.$ To be more
precise, if $p_{\bar{\varepsilon},\bar{c}}$ exists on $(-\infty,T]$ then in
some neighborhood \ \ of $\left(  \bar{\varepsilon},\bar{c}\right)  ,$
$p_{\varepsilon,c}\left(  t\right)  $ exists for $-\infty<t\leq T$ and is a
continuous function of $\left(  \varepsilon,c,t\right)  .$\footnote{See the
footnote at the end of the appendix for a further discussion of this point.}

From the third sentence of Lemma \ref{lem2} it follows that $c_{1}$ can be
chosen in $\left(  0,c_{0}^{\ast}\right)  $ such that if $c_{1}\leq
c<c_{0}^{\ast},$ then $v_{0,c}\left(  T\left(  c\right)  \right)
>q_{0}S\left(  u_{knee}\right)  ,$ where $T\left(  c\right)  $ is the unique
zero of $w_{0,c}.$ This proves the first assertion of Theorem \ref{thm2}. We
now choose $c_{1}$ in this way.

\begin{lemma}
\label{lem2a}There is an $\varepsilon_{0}>0$ and a $\tau>0$ such that if
$0<\varepsilon<\varepsilon_{0}$ and $c\in I=[c_{1},c_{0}^{\ast}+1]$, then
$u_{\varepsilon,c}\left(  t\right)  =u_{knee}$ for some first $t\leq\tau,$ and
$w_{\varepsilon,c}>0$ on $(-\infty,\tau].$ (Hence, $p_{\varepsilon,c}$
satisfies the the first condition of Theorem \ref{thm1}.) Further,
$\varepsilon_{0}$ can be chosen so that $p_{\varepsilon,c_{1}}$ satisfies
conditions (\ref{i}) and (\ref{ii}) of Theorem \ref{thm1}.
\end{lemma}

\begin{proof}
From the choice of $c_{1},$ Lemma \ref{lem2} implies that for any $c\geq
c_{1},$ there is a $\delta>0$ such that if $c\in I$ then $v_{0,c}^{\prime
}=w_{0,c}\geq\delta$ in the interval $\left[  0,\tau\right]  $ where
$u_{m}\leq u\leq u_{knee}.$ From $cu^{\prime}=v-u$ it follows that for some
$\tau>0,$ if $c_{1}\leq c\leq c_{0}^{\ast}$ then $u_{0,c}=u_{knee}$ before
$t=\tau.$ The uniform continuity of $u_{0,c}\left(  t\right)  $ in $\left(
\varepsilon,c,t\right)  ,$ for $-\infty<t\leq\tau$ and $d\in I$ in any compact
interval $[0,\hat{\varepsilon}]$ with $\hat{\varepsilon}>0,$ \ then implies
the first conclusion of the Lemma. \ The remaining assertion of Lemma
\ref{lem2a} follows by similar arguments.
\end{proof}

We have now proved Theorems \ref{thm1}, \ref{thm2}, and \ref{thm1a}, in that
order. To prove Theorem \ref{thm3} apply a continuity argument similar to that
just used to show that for any $\delta>0$ there is an $\varepsilon_{1}$ such
that if $\delta\leq c\leq c_{0}^{\ast}-\delta$ and $0<\varepsilon
<\varepsilon_{1},$ the pair $\left(  \varepsilon,c\right)  $ satisfy the
hypotheses on $\left(  \varepsilon,c_{1}\right)  $ in Theorem \ref{thm1}. It
follows that pulses exist for some $c_{\ast}\in\left(  0,\delta\right)  $ and
some $c^{\ast}\in\left(  c_{0}^{\ast}-\delta,c_{0}^{\ast}\right)  .$ But Lemma
\ref{lem10}\ implies that $\frac{\varepsilon}{c_{\ast}}<M.$ Theorem \ref{thm3} follows.

\section{Discussion\label{discussion}}

\subsection{Hypothesis 3.1 of \cite{faye}}

As stated earlier, there is an additional hypothesis in the existence result
given in \cite{faye}, namely Hypothesis 3.1 in that paper. This hypothesis is
interesting in a broader context, and we will include some comments on its
relation to the well-known pde model of FitzHugh and Nagumo.

To state this hypothesis we need to introduce a basic tool in the method of
geometric perturbation, the so-called \textquotedblleft
singular\textquotedblright\ solution. The singular solution of (\ref{1}) is a
continuous piecewise smooth curve in $R^{4}$ consisting of four smooth pieces.

The first piece is the front with speed $c_{0}^{\ast}$ found in Lemma
\ref{lem1}. (Recall that \textquotedblleft fronts\textquotedblright\ were
defined just after the statement of this lemma.) In Figure \ref{figfront} the
green line segment is the projection of the graph of the front onto the
$\left(  u,q\right)  $ plane.

\begin{figure}[h]
\includegraphics[height=1.6 in, width =1.2  in]{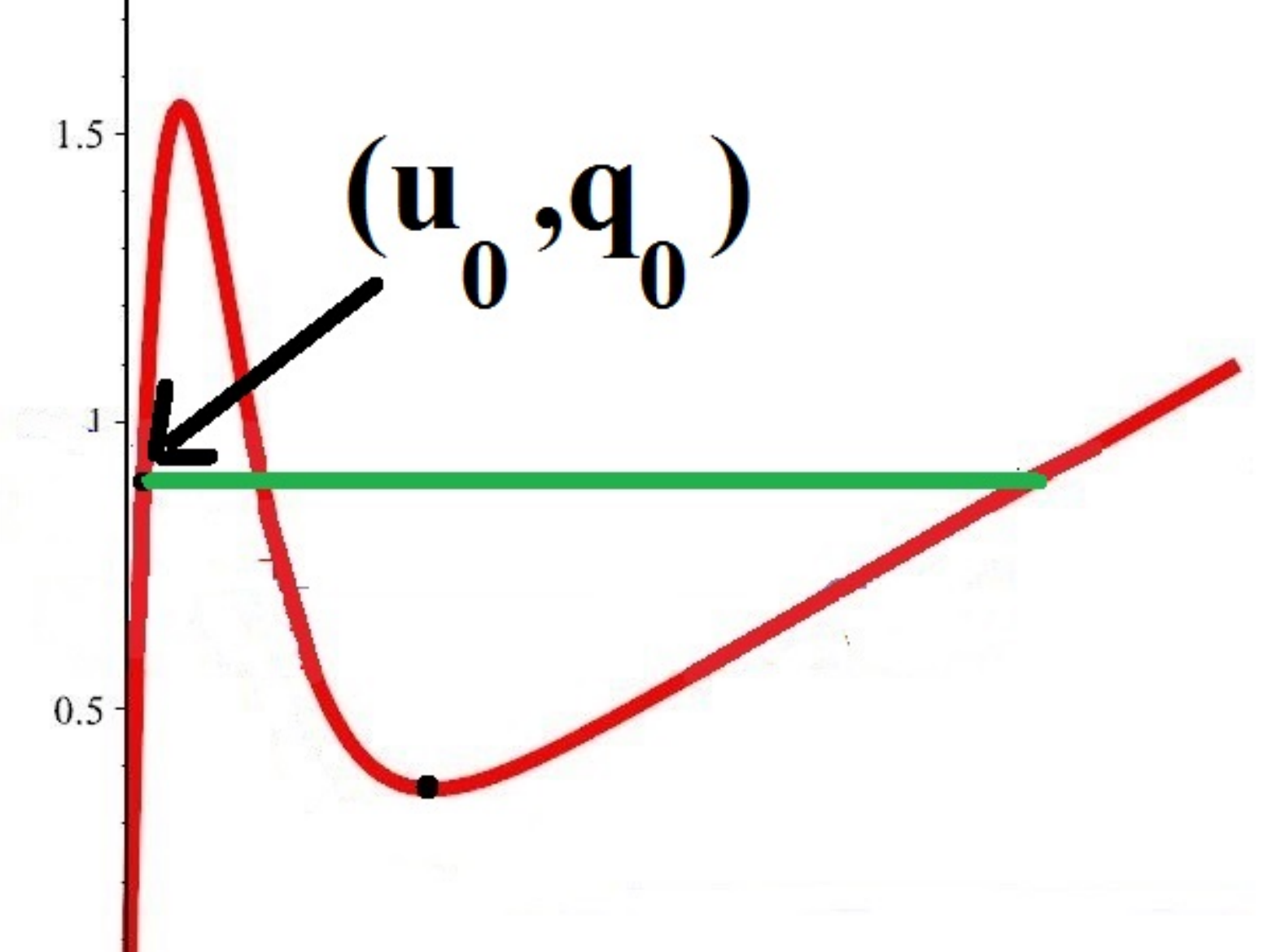}\caption{The red
curve is the graph of $q=\frac{u}{S\left(  u\right)  }$}%
\label{figfront}%
\end{figure}

The second piece of the singular solution is a segment of the nullcline
$u=qS\left(  u\right)  $ (with $w=0,$ $v=u$) as shown in Figure \ref{figslow}.
It is obtained from (\ref{1}) by letting $p\left(  t\right)  =P\left(
\varepsilon t\right)  ,$ formally setting $\varepsilon=0$ in the resulting
system of ode's for $P=\left(  U,V,W,Q\right)  ,$ and solving the resulting
set of one differential equation and three algebraic equations, one of which
is $U-QS\left(  U\right)  =0.$ For more information on this segment, and the
singular solution in general, see \cite{faye}. \ We don't need to say more
about this segment here.

\begin{figure}[h]
\includegraphics[height=1.6 in, width =1.2  in]{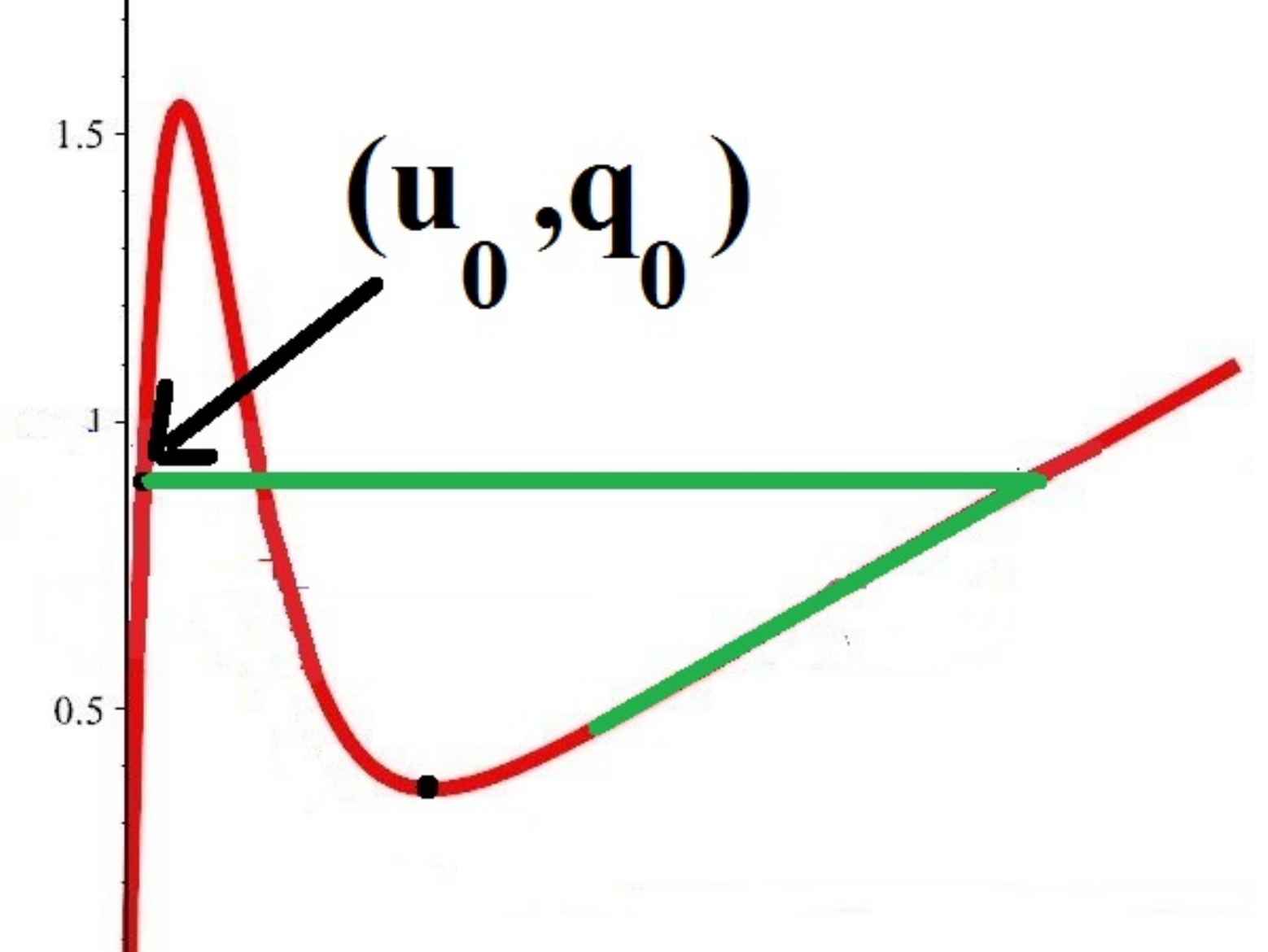}\caption{The first
two segments of the singular solution}%
\label{figslow}%
\end{figure}

To define the third part of the singular solution (crucial in Hypothesis 3 of
\cite{faye}), we consider the fast system (\ref{2}), but with $q_{0}$ replaced
by $q,$ as a parameter ranging between $q_{\min}$ and $q_{0}.$%

\begin{equation}
\left.
\begin{array}
[c]{c}%
u^{\prime}=\frac{v-u}{c}\\
v^{\prime}=w\\
w^{\prime}=b^{2}\left(  v-qS\left(  u\right)  \right)
\end{array}
\right.  \label{fast}%
\end{equation}
\bigskip

For each $q\in(q_{\min},q_{0}]$ there is a unique $c\left(  q\right)  \geq0$
such that (\ref{fast}) has a bounded non-constant solution
\[
r_{c\left(  q\right)  }=\left(  u_{c\left(  q\right)  },v_{c\left(  q\right)
},w_{c\left(  q\right)  }\right)  .
\]
The graph of this solution is a heteroclinic orbit connecting the left and
right branches of $q=\frac{u}{S\left(  u\right)  }$ in the $\left(
u,q\right)  $ plane. There is at least one value of $q\in\left(  q_{\min
},q_{0}\right)  $ such that $c\left(  q\right)  =0.$ This will be true for any
$q$ such that
\[
\int_{u_{-}\left(  q\right)  }^{u_{+}\left(  q\right)  }\left(  qS\left(
u\right)  -u\right)  du=0.
\]
For $q\in\left(  q_{\min},q_{0}\right)  $ and sufficiently close to $q_{\min}%
$, the integral above is negative, and whenever this is the case, the
connecting heteroclinic orbit exists for some $c\left(  q\right)  >0,$ but
unlike the front defined earlier, $u_{c\left(  q\right)  }$ is decreasing,
from $u_{+}\left(  q\right)  $ to $u_{-}\left(  q\right)  .$ Such a solution
is called a \textquotedblleft back\textquotedblright. \ If $q=q_{\min}$ then
there is a \textquotedblleft back\textquotedblright\ for any $c\geq
\lim_{q\rightarrow q_{\min}^{+}}c\left(  q\right)  .$ \ All this can be proved
using methods from the appendix, or see \cite{faye}.

\bigskip

The third part of the singular solution of (\ref{1}) is a \textquotedblleft
back\textquotedblright, at a value $q=q_{1}\in\lbrack q_{\min},q_{0})$\ with
speed $c=c\left(  q_{0}\right)  .$ There must be at least one $q_{1}$ for
which such a back exists. As far as we know there is no proof that $c$ is a
monotone function of $q$ in $\left(  q_{\min},q_{0}\right)  ,$ so possibly
there could be more than one such $q_{1}.$ In this case we can require that
the jump down is at the largest possible $q_{1}$ supporting a back with speed
$c\left(  q_{0}\right)  .$ The singular solution is said to \textquotedblleft
jump down above the knee\textquotedblright\ if $q_{1}>q_{\min}.$

Numerical computations suggest that often no such $q_{1}$ exists in $\left(
q_{\min},q\right)  $. In this case, there is still a traveling back with speed
$c\left(  q_{0}\right)  ,$ but it is at $q_{1}=q_{\min}.$

\begin{figure}[h]
\includegraphics[height=1.6 in, width =1.2  in]{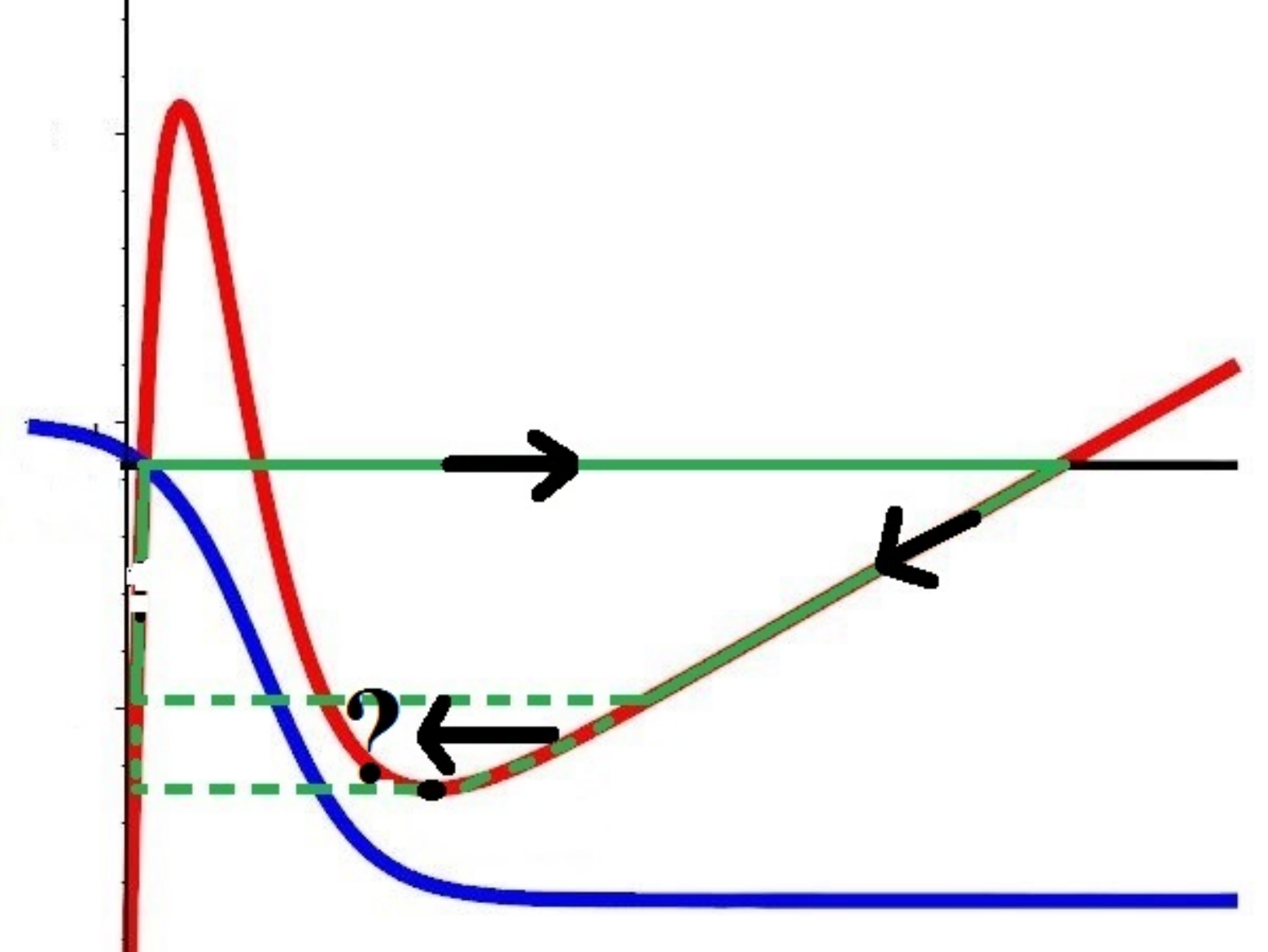}\caption{Nullclines:
Red, Blue; Singular solution: Green}%
\end{figure}

The fourth part of the singular solution is a slow return along the left
branch of the nullcline.

\bigskip

Hypotheses of Faye:\medskip

(i) \textit{The system (\ref{1}) has a unique equilibrium point }$\left(
u_{0},u_{0},0,q_{0}\right)  $.\textit{ }

\textit{(ii) If }$h\left(  u\right)  =\frac{u}{S\left(  u\right)  }$\textit{
then the equation }%
\[
g\left(  u\right)  =q_{0}%
\]
\textit{has exactly three solutions, }$u_{0}<u_{m}<u_{+},$\textit{ with
}$h^{\prime}\left(  u_{0}\right)  >0,$\textit{ }$h^{\prime}\left(
u_{m}\right)  <0,$\textit{ and }$h^{\prime}\left(  u_{+}\right)  >0.$\textit{
}

\textit{(iii)}%
\[
\int_{u_{0}}^{u_{+}}\left(  q_{0}S\left(  u\right)  -u\right)  ~du>0.
\]

\textit{(iv) The speed of any \textquotedblleft back\textquotedblright\ with
}$q\in(q_{\min},q_{0})$ \textit{is less than }$c\left(  q_{0}\right)
.\medskip$

It follows that the back of the singular solution is required to be at the
knee. This condition can only be verified by numerical integration of the fast system.

\begin{theorem}
\textbf{(}Faye\textbf{): }\textit{Under hypotheses (i), (ii), (iii), and (iv),
if }$\varepsilon$ is sufficiently small then \textit{the system (\ref{1}) has
a homoclinic orbit for at least one positive value of }$c.$
\end{theorem}

Geometrical perturbation, based on work of Fenichel and others, is a technique
for showing that the singular solution is close to a real solution if
$\varepsilon$ is sufficiently small. Certain \textquotedblleft
transversality\textquotedblright\ conditions can be complicated to check,
requiring a technique called \textquotedblleft blow-up\textquotedblright.

\bigskip

We can contrast the existence of a back at the knee with the well-known
behavior of the pde model of FitzHugh and Nagumo. \ (See \cite{hmc} for a
presentation of this model and a proof that it has two traveling pulses. ) The
fast FitzHugh-Nagumo pulse can be described as having a jump up, (close to a
front) during which $u$ increases rapidly while $w$ is nearly zero, followed
by a slow increase in $w,$ and then a jump down (close to a back), with $w$
again nearly constant (but positive), and $u$ decreasing rapidly. In this
case, the back occurs before $\left(  u,w\right)  $ reaches the knee. See
\ref{figure5}.

\begin{figure}[h]
\includegraphics[height=1.5 in, width =2  in]{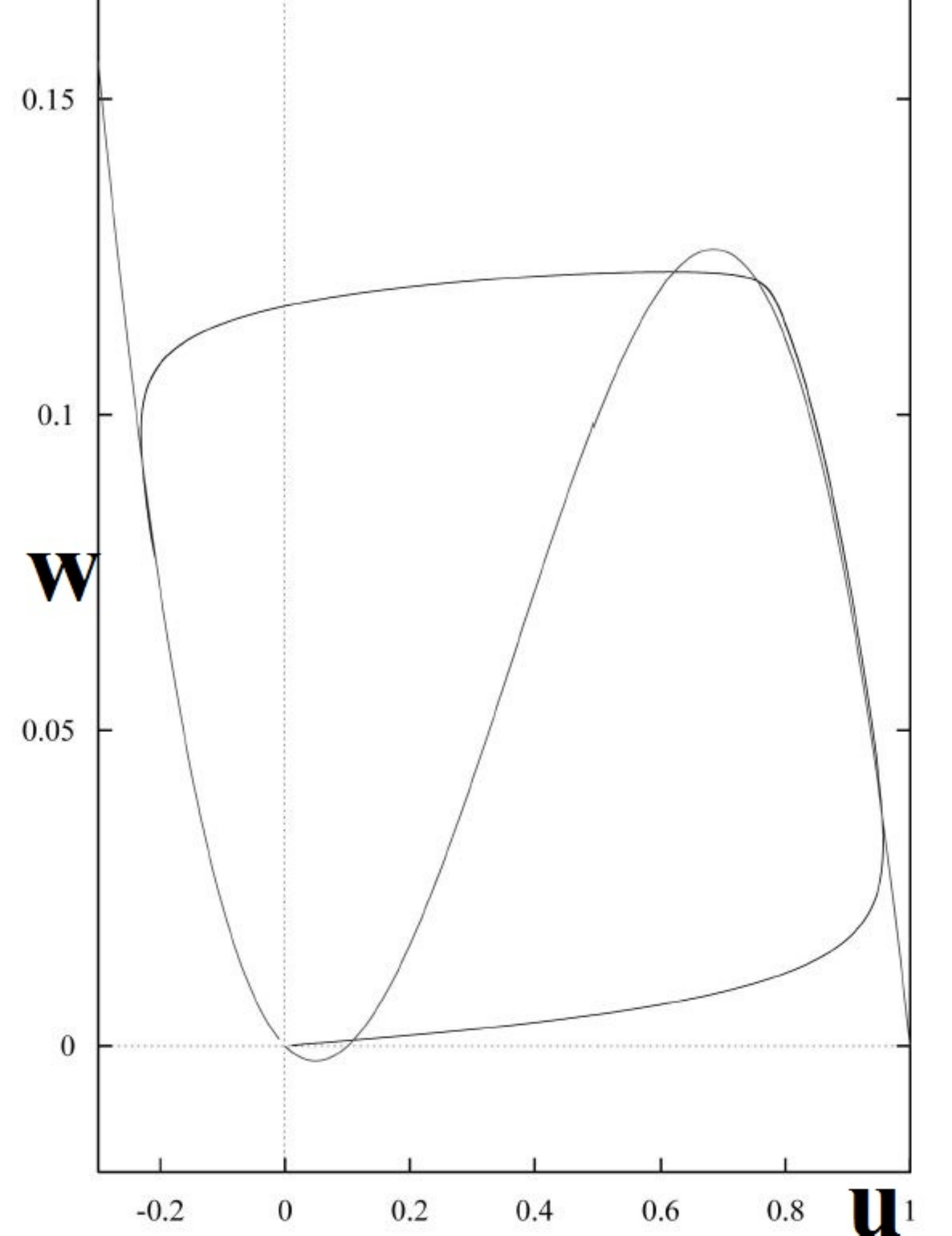}\caption{FitzHugh-Nagumo
pulse ($u,v$ plane)}%
\label{figure5}%
\end{figure}

We have done some preliminary numerical investigation to test whether it is
possible, in the model studied in this paper and with the particular function
$S$ in (\ref{-3}), to adjust the parameters $\lambda$ and $\kappa$ so that the
jump down occurs before $w$ reaches the knee. \ We have not found such a pair
$\left(  \lambda,\kappa\right)  ,$ but we cannot assert that none exists.

For the FitzHugh-Nagumo model, however, it is clear that the jump down is
always before reaching the knee. \ This follows because the reaction term in
equations is a cubic polynomial, $f\left(  u\right)  =u\left(  1-u\right)
\left(  u-a\right)  ,$ where $0<a<\frac{1}{2}.$ This function is symmetric
around its inflection point, which leads to the "before the knee" behavior of
the singular solution. \ 

So we searched numerically for alternative functions to use for $f$ which,
while still \textquotedblleft cubic like\textquotedblright,\thinspace\ permit
the down jump of the singular solution to be at the knee. We found such a
function, as illustrated in Figure \ref{figure6}. \ We are not aware of a
method which determines analytically where the downjump occurs, either for the
model of Faye or that of FitzHugh-Nagumo when $f$ is asymmetric. \ 

\begin{figure}[h]
\includegraphics[height=1.5 in, width =3  in]{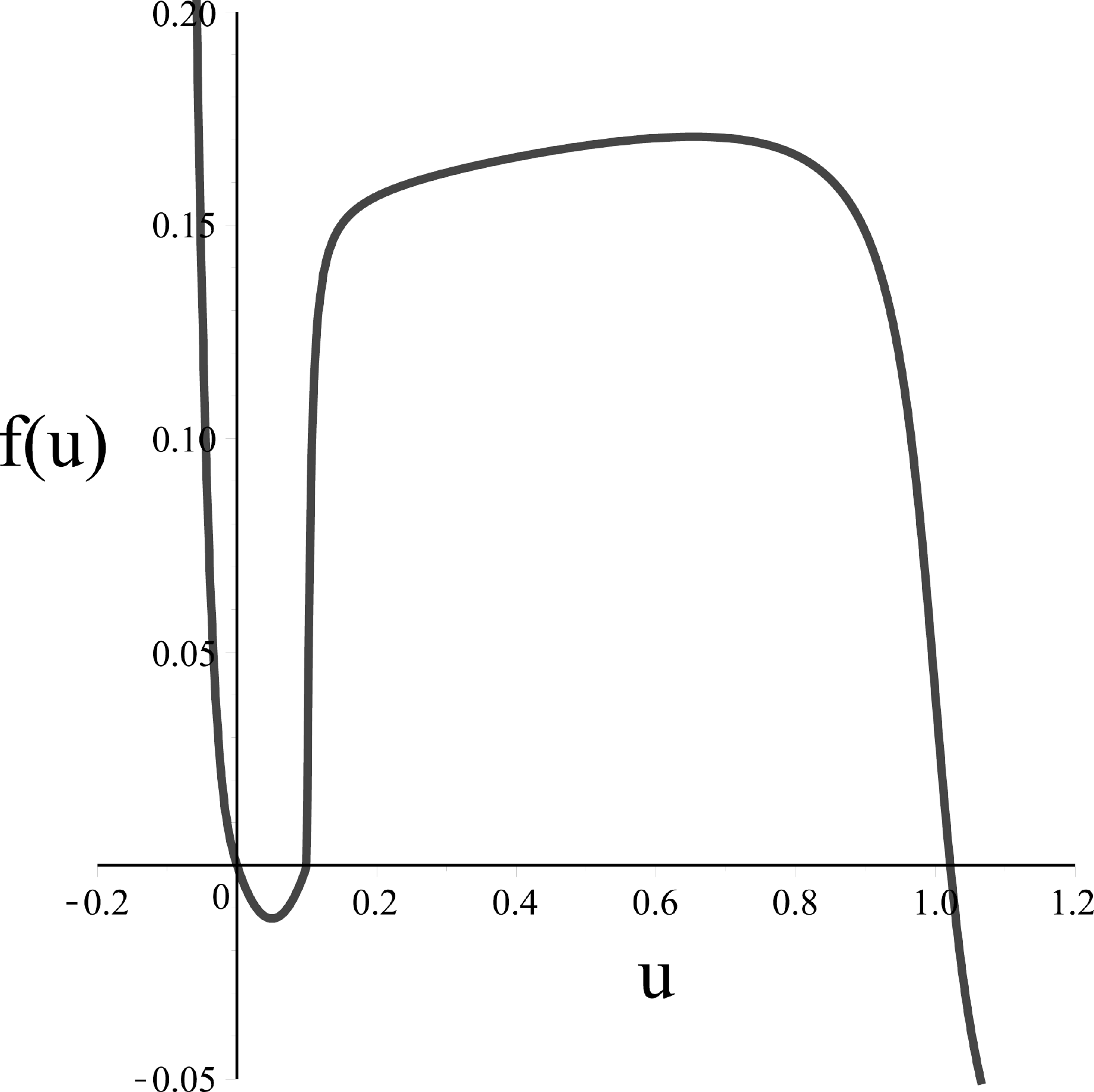}\caption{a
"cubic-like" function for an alternative FitzHugh-Nagumo type model}%
\label{figure6}%
\end{figure}

The existence proof in \cite{hmc} for fast and slow homoclinic orbits of the
FitzHugh-Nagumo system also applies to functions $f$ such as that pictured in
Figure \ref{figure6}. On the other hand it appears that the proof by geometric
perturbation in this case, while probably still basically valid, requires a
more complicated analysis because the downjump of the singular solution may
occur at the knee.

\subsection{The slow pulse}

It appears that the fast-slow analysis used to get the singular solution, for
any of the models we have discussed, does not apply to the slow pulse. Hence
it appears difficult to use geometric perturbation to obtain this solution.
\ In [\cite{kss} this solution was obtained in the FitzHugh-Nagumo case using
dynamical systems methods, but only for $a$ sufficiently close to $\frac{1}%
{2}.$

The analysis in \cite{kss} was in the region $a$ close to $\frac{1}{2},$
$\varepsilon$ and $c$ small. \ Indeed, if $a=\frac{1}{2}$ then there is only
one pulse and it is a standing wave ($c=0$). \ It is proved in \cite{kss} that
for $\frac{1}{2}-a$ positive but small there is a smooth curve $\left(
c,\varepsilon\left(  c\right)  \right)  $ corresponding to pulse solutions and
connecting $c_{\ast}$ to $c^{\ast}.$ The speeds $c_{\ast}$ and $c^{\ast},$ and
the maximum of $\frac{\varepsilon\left(  c\right)  }{c}$ in $[c_{\ast}%
,c^{\ast}]$ all tend to zero as $a\rightarrow\frac{1}{2}^{-}.$ \ \ 

We suspect that a similar picture holds for the model of Faye, but it is not
clear that our analysis is able to prove this much. \ \ The FitzHugh-Nagumo
condition that $a-\frac{1}{2}$ is small would be replaced here by requiring
that $\int_{u_{0}}^{u_{+}}\left(  q_{0}S\left(  u\right)  -u\right)  du$ is
small. \ 

\subsection{Model of Pinto and Ermentrout}

We said in the introduction that there was only a partial existence result for
the model in \cite{pe}. We were referring there to a recent paper by Faye and
Scheel \cite{fs}, which uses an interesting extension of the geometric
perturbation method to infinite dimensional spaces to handle this kind of
problem. \ The method is powerful because it allows an extension beyond the
sorts of kernel which reduce the problem to an ode.

Faye and Scheel remark that their paper appears to apply to the
Pinto-Ermentrout model. This is true, but with a limitation. A key hypothesis
in their paper is that for the singular solution, as described above, the jump
down occurs above the knee. In a private communication Professor Ermentrout
has observed that while this is true for the Pinto-Ermentrout model in some
parameter ranges, it is also common for the down-jump to occur at the knee.
\ This is why we characterized their result as "partial\textquotedblright\ for Pinto-Ermentrout.

Unfortunately, we have not been able to make our approach work for
Pinto-Ermentrout. The reason may be related to an important difference between
(\ref{1}) and the equivalent set of ode's obtained from (\ref{pe}). The
linearization of (\ref{1}) around its equilibrium point has only real
eigenvalues, for any $\varepsilon>0,$ while the equivalent linearization for
(\ref{pe}) has complex eigenvalues for a range of positive $\varepsilon$.
Thus, a homoclinic orbit would oscillate around equilibrium. \ A few
oscillations could occur even for the very small values of $\varepsilon$ where
the eigenvalues are real. The final steps in our proof above clearly do not
allow such oscillations.

In [\cite{h1}] we showed that there was co-existence of complex roots and a
homoclinic orbit for FitzHugh-Nagumo, and observed that work of Evans,
Fenichel and Feroe then implied the existence of many periodic solutions and a
form of chaos. This leads to a conjecture that the Pinto-Ermentrout model
supports a richer variety of bounded solutions than the model of Faye. The new
solutions are probably unstable, however, so their physical importance is unclear.

\subsection{Stability of these solutions}

A local stability result for the fast solution was proved by Faye. His proof
depends crucially on analysis of both the front and back of his solution, as
described above. The analysis of the back is less standard because of the
assumption that the jump down is at the knee. \ Here he relies on previous
work on similar problems. Presumably if the jump were above the knee (where,
however his existence proof is not claimed to apply), the stability analysis
would be easier.

\subsection{Can the hypotheses of Theorem \ref{thm1} be checked rigorously for
a specific $\left(  \varepsilon,c_{1}\right)  $?}

Condition (\ref{i}) says that the solution $p_{\varepsilon,c_{1}}$ is on the
unstable manifold $\mathcal{U}_{\varepsilon,c_{1}}^{+}$ of (\ref{1}) at the
equilibrium point $p_{0}.$ To check condition (\ref{ii}) we must follow
$p_{\varepsilon,c_{1}}$ until a point where $u_{\varepsilon,c_{1}}=0.$ Our
proposal for doing this is based on \cite{hassard}, where a similar procedure
was followed for the well known equations of Lorenz.

Using a standard ode solver we can arrive at a conjectured value for $\left(
\varepsilon,c_{1}\right)  .$ To begin analyzing $p_{\varepsilon,c_{1}}$
numerically we would expand the solutions around $p_{0}.$ A high order
expansion of $\mathcal{U}_{\varepsilon,c}^{+}$ results in algebraic
expressions which are then evaluated using rigorous numerical analysis based
on interval arithmetic. With this technique one hopes to show that
$\mathcal{U}_{\varepsilon,c}^{+}$ enters a very small box near $p_{0}$. For
the example in \cite{hassard} this box had a diameter of about $10^{-68}.$
This gives us an initial estimate accurate to (say) 68 significant digits.

From there, a rigorous ode solver, as described for example in \cite{mayer},
would be used to continue $p_{\varepsilon,c_{1}}$ until $w_{\varepsilon,c_{1}%
}=0.$ Whether this can be done cannot be determined ahead of time. \ One has
to run the solver. \ The number of guaranteed accurate digits decreases as the
integration proceeds.\ We then hope that some significant digits would be
maintained long enough to reach $u=0.$ It must be checked along the way that
$\max u_{\varepsilon,c_{1}}>u_{knee}.$

Based on the great sensitivity of the Lorenz equations to the initial
conditions, we expect that this would be easier for the Faye model than it was
in \cite{hassard}. We are not aware of any proposal to try to estimate
$\varepsilon$ for the method of geometric perturbation.

\appendix{}

\section{Proofs of Lemmas \ref{lem1}, \ref{lem2}, and \ref{lem10}.}

\subsection{Proofs of Lemma \ref{lem1} and Lemma \ref{lem2}}

We prove these results together. The linearization of (\ref{2}) around
$r_{0}=\left(  u_{0},u_{0},0\right)  $ is the system $Q^{\prime}=AQ$ with
\[
A=\left(
\begin{array}
[c]{ccc}%
-\frac{1}{c} & \frac{1}{c} & 0\\
0 & 0 & 1\\
-b^{2}q_{0}S^{\prime}\left(  u_{0}\right)  & b^{2} & 0
\end{array}
\right)  .
\]
The characteristic polynomial of $A$ is
\begin{equation}
f\left(  X\right)  =X^{3}+\frac{1}{c}X^{2}-b^{2}X-\frac{b^{2}}{c}(1-S^{\prime
}\left(  u_{0}\right)  q_{0}), \label{17}%
\end{equation}
Recall that $h(u)=\frac{u}{S(u)}$. Condition \ref{c2} implies that the
equation $q_{0}=h\left(  u\right)  $ has three solutions, $u_{0}<u_{m}<u_{+},$
and by Condition \ref{c0a}, $h^{\prime}\left(  u_{0}\right)  >0.$ It follows
that
\begin{equation}
q_{0}S^{\prime}\left(  u_{0}\right)  <1. \label{17a}%
\end{equation}
Therefore $f\left(  0\right)  <0.$ Also, $f^{\prime}\left(  0\right)
=-b^{2},$ and both $f^{\prime\prime}$ and $f^{\prime\prime\prime}$ are
positive for $X>0.$ Hence $A$ has one real positive eigenvalue. \ Also,
$f\left(  -\frac{1}{c}\right)  =\frac{b^{2}}{c}S^{\prime}\left(  u_{0}\right)
q_{0}>0,$ which implies that $A$ has two real negative eigenvalues.

Further, it is easily seen that there is an eigenvector corresponding to the
positive eigenvalue of $A$ which points into the positive octant. If
$r=\left(  u,v,w\right)  $ is a solution lying on the branch $\mathcal{U}%
_{0,c}^{+}$ of the unstable manifold of (\ref{2}) at $r_{0}$, then initially,
$u^{\prime},$ $v^{\prime}=w,$ and $w^{\prime}$ are positive. It follows from
the first two equations of (\ref{2}) that $v>u$ as \ long as $w>0$. \ Also,
$u\geq q_{0}S\left(  u\right)  $ for $u_{0}\leq u\leq u_{m}$ and so
$w^{\prime}>0$ while $u$ is in $\left(  u_{0},u_{m}\right]  .$ Hence there is
a first $t_{0}$ such that $u\left(  t_{0}\right)  =u_{m},$ and we can assume
that $t_{0}=0.$ We have now proved the assertions of the first and second
sentences of Lemma \ref{lem1}.

For the third sentence of Lemma \ref{lem1}, and for all of Lemma \ref{lem2},
we need a comparison lemma. For each $c>0,$ let $r_{c}=\left(  u_{c}%
,v_{c},w_{c}\right)  $ be the unique solution of (\ref{2}) on $\mathcal{U}%
_{0,c}^{+}$ such that $w_{c}>0$ on $(-\infty,0]$ and $u_{c}\left(  0\right)
=u_{m}.$ Suppose that $w_{c}>0$ on a maximal interval $(-\infty,T\left(
c\right)  ),$ where possibly $T\left(  c\right)  =\infty.$ Then in
$(-\infty,T\left(  c\right)  )$ we can consider $u$ and $w$ as functions of
$v$, letting $u_{c}\left(  t\right)  =U_{c}\left(  v_{c}(t\right)  )$ and
$w_{c}\left(  t\right)  =W_{c}\left(  v_{c}\left(  t\right)  \right)  .$ This
defines the functions $U_{c}$ and $W_{c}$ on the interval
\[
I_{c}=(u_{0},\lim_{t\rightarrow T\left(  c\right)  ^{-}}v_{c}\left(  t\right)
),
\]
and for $v$ in this interval,
\begin{equation}
U_{c}^{\prime}\left(  v\right)  =\frac{v-U_{c}\left(  v\right)  }%
{cW_{c}\left(  v\right)  },~W_{c}^{\prime}\left(  v\right)  =\frac
{b^{2}\left(  v-q_{0}S\left(  U_{c}\left(  v\right)  \right)  \right)  }%
{W_{c}\left(  v\right)  }. \label{a01}%
\end{equation}

\begin{lemma}
\label{lem12} If $d_{2}>d_{2}>0$ \ then $I_{d_{1}}\subset I_{d_{2}}.$ In the
interval $I_{d_{1}}$, \
\begin{equation}
\left.
\begin{array}
[c]{c}%
U_{d_{2}}<U_{d_{1}}\\
W_{d_{2}}>W_{d_{1}}%
\end{array}
\right.  . \label{a1a}%
\end{equation}

\end{lemma}

\begin{proof}
The first sentence follows by proving (\ref{a1a}) on the smaller of the two
intervals. We first show that these inequalities hold on some initial interval
$u_{0}<v<u_{0}+\delta$. This is seen by comparing unit eigenvectors
corresponding to the positive eigenvalues $\lambda_{1}\left(  d_{1}\right)  $
and $\lambda_{1}\left(  d_{2}\right)  $ of the linearizations of (\ref{2})
around $r_{0}$. Suppose that for a particular $c$ the eigenvector
corresponding to $\lambda_{1}\left(  c\right)  $ is $\left(  n_{1}\left(
c\right)  ,n_{2}\left(  c\right)  ,n_{3}\left(  c\right)  \right)  .$ Then
\begin{align*}
n_{1}\left(  c\right)   &  =\frac{n_{2}\left(  c\right)  }{\left(
1+\lambda_{1}\left(  c\right)  \right)  }.\\
n_{3}\left(  c\right)   &  =\lambda_{1}\left(  c\right)  n_{2}\left(
c\right)  .
\end{align*}
Inequalities (\ref{a1a}) follow near $r_{0}$ if $\lambda_{1}\left(
d_{2}\right)  >\lambda_{1}\left(  d_{1}\right)  .$ For this we turn to the
characteristic polynomial of $A,$ given in (\ref{17}) but now denoted by
$f\left(  X,c\right)  $.

It is easier to work with $F=cf,$ noting that $c>0.$ The positive eigenvalue
of $A$ is determined by the equation
\[
F\left(  \lambda_{1}\left(  c\right)  ,c\right)  =0
\]
and the condition $\lambda_{1}\left(  c\right)  >0.$ Then
\[
\frac{\partial F}{\partial X}\left(  \lambda_{1}\left(  c\right)  ,c\right)
\frac{d\lambda_{1}\left(  c\right)  }{dc}=-\frac{\partial F}{\partial
c}\left(  \lambda_{1}\left(  c\right)  ,c\right)  .
\]

Since $F\left(  0,c\right)  <0,$ $\frac{\partial F}{\partial X}\left(
0,c\right)  <0$ and $\frac{\partial^{2}F}{\partial X^{2}}\left(  X,c\right)
>0$ for $X\geq0,$ $\frac{\partial F}{\partial X}\left(  \lambda_{1}\left(
c\right)  ,c\right)  >0$. $\ $Also, $\frac{\partial F}{\partial c}\left(
\lambda_{1},c\right)  =$ $\lambda_{1}\left(  \lambda_{1}^{2}-b^{2}\right)  .$
It follows that $\frac{d\lambda_{1}\left(  c\right)  }{dc}>0$ if $\lambda
_{1}<b.$ But
\[
F\left(  b,c\right)  =\frac{b^{2}}{c}S^{\prime}\left(  u_{0}\right)  q_{0}>0,
\]
so indeed, $\lambda_{1}\left(  c\right)  <b.$

Therefore (\ref{a1a}) holds on some interval $\left(  u_{0},u_{0}%
+\delta\right)  .$ Suppose that the first inequality fails at a first $\hat
{v}\in I_{d_{1}},$ while the second holds over $(u_{0},\hat{v}].$ \ Then at
$\hat{v}$ , $U_{1}=U_{2},$ $W_{2}>W_{1}.$ But then, $U_{2}^{\prime}\left(
\hat{v}\right)  <U_{1}^{\prime}\left(  \hat{v}\right)  ,$ a contradiction
since $U_{2}<U_{1}$ on $(0,\hat{v}].$ A similar argument eliminates the other
possibilities, using the fact that $S$ is increasing, and this completes the
proof of the Lemma \ref{lem12}.
\end{proof}

\begin{corollary}
\label{cor1}If $T\left(  d_{1}\right)  =\infty$ then $T\left(  d_{2}\right)
=\infty.$
\end{corollary}

\begin{lemma}
\label{lem12b}If $d_{2}\geq d_{1}$, $T\left(  d_{1}\right)  <\infty,$ and
$v_{d_{1}}\left(  T\left(  d_{1}\right)  \right)  >q_{0}S\left(
u_{knee}\right)  ,$ then either $T\left(  d_{2}\right)  =\infty$ or $u_{d_{2}%
}\left(  T\left(  d_{2}\right)  \right)  >u_{knee}.$
\end{lemma}

\begin{proof}
Suppose that $T\left(  d_{2}\right)  <\infty$ and $u_{d_{2}}\left(  T\left(
d_{2}\right)  \right)  \leq u_{knee}.$ Lemma \ref{lem12} implies that
$v_{d_{2}}\left(  T\left(  d_{2}\right)  \right)  \geq v_{d_{1}}\left(
T\left(  d_{1}\right)  \right)  $ and so
\[
w_{d_{2}}^{\prime}\left(  T\left(  d_{2}\right)  \right)  =b^{2}\left(
v_{d_{2}}\left(  T\left(  d_{2}\right)  \right)  -q_{0}S\left(  u_{d_{2}%
}\left(  T\left(  d_{2}\right)  \right)  \right)  \right)  \geq b^{2}\left(
v_{d_{2}}\left(  T\left(  d_{2}\right)  \right)  -q_{0}S\left(  u_{knee}%
\right)  \right)  >0,
\]
a contradiction because $T\left(  d_{2}\right)  $ is the first zero of
$w_{d_{2}}.$ \ 
\end{proof}

Next we must show the existence of $c_{0}^{\ast}.$

\begin{lemma}
\label{lem13a} For sufficiently large $c,$ $T\left(  c\right)  =\infty.$
\end{lemma}

\begin{proof}
From (\ref{a01}), in the interval $\left(  -\infty,T\left(  c\right)  \right)
,$ where $v_{c}$ is increasing, $v_{c}>u_{c}.$ Recall that
\[
\frac{u}{S\left(  u\right)  }>q_{0}%
\]
in $\left(  u_{0},u_{m}\right)  .$ As long as $u_{0}<U_{c}\left(
v_{c}\right)  <u_{m}$ and $v_{c}<1,$
\[
\frac{dU_{c}}{dW_{c}}=\frac{v-U_{c}}{cb^{2}S\left(  U_{c}\right)  \left(
\frac{v}{S\left(  U_{c}\right)  }-q_{0}\right)  }<\frac{1}{c\eta\left(
\frac{U_{c}}{S\left(  U_{c}\right)  }-q_{0}\right)  }.
\]

This implies that for large $c$, in the interval where $u_{0}<u_{c}<u_{m},$
$w_{c}$ grows rapidly and so, in turn, does $v_{c.}$. \ In particular,
$v_{c}>1$ before $u_{c}=u_{m},$ and this implies that $T\left(  c\right)
=\infty.$
\end{proof}

Now we wish to show that for small $c>0,$ $w_{c}<0$ before $v_{c}=1.$ It is in
this step that Condition \ref{c3} is used.

\begin{lemma}
\label{lem14} There is a $\bar{w}>0$ such that for any $c>0,$ if
$|w_{c}\left(  \tau\right)  |>\bar{w}$ and $0<v_{c}\left(  \tau\right)  <1,$
then $|w_{c}|>\bar{w}$ for $t>\tau$ and $v_{c}$ leaves the interval $(0,1)$.
If $w_{c}(\tau)>\bar{w}$ then $v_{c}$ crosses $1$, while if $w_{c}(\tau
)<-\bar{w}$ then $v_{c}$ crosses $0$.
\end{lemma}

\begin{proof}
Let $\bar{w}=\sqrt{2}b$. Since $\left\vert w^{\prime}\right\vert \leq b^{2},$
if $w\left(  \tau\right)  =\sqrt{2}b,$ then for $s>0,$ $w\left(
\tau+s\right)  \geq\sqrt{2}b-b^{2}s,$ from which follows that $v$ must leave
$\left(  0,1\right)  $ before $s=\frac{\sqrt{2}}{b}.$ \ 
\end{proof}

\begin{lemma}
\label{lem15}If $0<w_{c}\leq\bar{w}$ on $(-\infty,\tau]$ then $0<v_{c}%
-u_{c}<c\bar{w}$ on this interval. If $\left\vert w_{c}\right\vert \leq\bar
{w}$ on $(-\infty,\sigma],$ then $\left\vert v_{c}-u_{c}\right\vert <c\bar{w}$
on this interval.
\end{lemma}

\begin{proof}
With $r=r_{c}$, $\left(  v-u\right)  ^{\prime}=w-\frac{v-u}{c}\leq\bar
{w}-\frac{v-u}{c},$ so if $v-u>c\bar{w}$ then $\left(  v-u\right)  ^{\prime
}<0.$ Also, if $v-u=0$ in $\left(  -\infty,\tau\right)  $ then$\left(
v-u\right)  ^{\prime}>0.$ Since $v-u\rightarrow0^{+}$ as $t\rightarrow-\infty$
, the first sentence of the lemma follows and the second is similar.
\end{proof}

Based on this lemma, we consider, in addition to (\ref{2}), the system%
\begin{equation}
\left.
\begin{array}
[c]{c}%
v^{\prime}=w\\
w^{\prime}=b^{2}\left(  v-q_{0}S\left(  v\right)  \right)
\end{array}
\right.  , \label{10}%
\end{equation}
This system has equilibrium points at $\left(  u_{0},0\right)  ,$ $\left(
u_{m},0\right)  ,$ and $\left(  u_{+},0\right)  ,$ and a standard phase plane
analysis, assuming Condition \ref{c3}, shows that the positive branch
$\mathcal{U}_{0,0}^{+}$ of unstable manifold of (\ref{10}) at $\left(
u_{0},0\right)  $ is homoclinic. Also we consider the system
\begin{equation}
\left.
\begin{array}
[c]{c}%
v^{\prime}=w\\
w^{\prime}=b^{2}\left(  v-q_{0}S\left(  v-\hat{c}\right)  \right)
\end{array}
\right.  , \label{11}%
\end{equation}
for small $\hat{c}.$ Choose $\hat{c}$ so small that this system also has three
equilibrium points, and a homoclinic orbit based at the left most of these.
\ This orbit entirely encloses the homoclinic orbit of (\ref{10}). \ 

Finally we consider the system
\begin{equation}
\left.
\begin{array}
[c]{c}%
v^{\prime}=w\\
w^{\prime}=b^{2}\left(  v-q_{0}S\left(  v+\hat{c}\right)  \right)
\end{array}
\right.  , \label{12a}%
\end{equation}

For sufficiently small $\hat{c}$ this system also has a homoclinic orbit.
\ This orbit lies entirely inside the homoclinic orbit of (\ref{10}). However,
the lower left branch $\mathcal{U}_{0,0}^{-}$ of the unstable manifold of this
system crosses the homoclinic orbits of (\ref{10}) and (\ref{11}), and this
branch will play a role below. \ (See Figure 12.)

\begin{figure}[h]
\includegraphics[height=2 in, width =5  in]{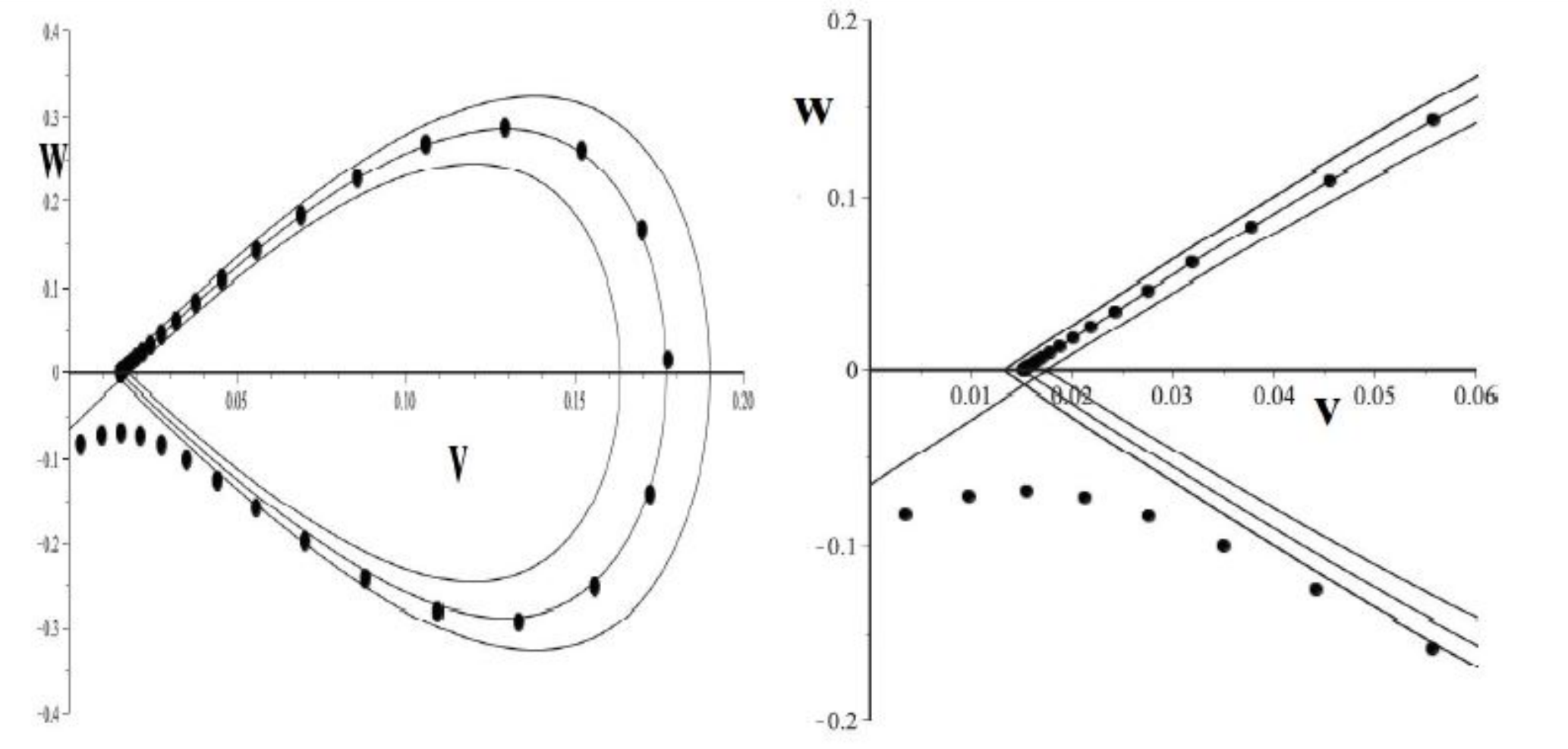} \label{figure77}%
\caption{ homoclinic orbits of, from inner to outer,(\ref{12a}), (\ref{10}),
and (\ref{11}), part of $\mathcal{U}_{0,0}^{-}$ for (\ref{12a}), and an orbit
of (\ref{2}) (dotted)}%
\end{figure}

From now on, $\left(  v_{1},w_{1}\right)  ,$ $\left(  v_{2},w_{2}\right)  ,$
and $\left(  v_{3},w_{3}\right)  $ will denote the unique solutions of the
systems (\ref{10}), (\ref{11}), and (\ref{12a}) respectively which lie on the
homoclinic orbits of those systems and satisfy $v_{i}\left(  0\right)
=u_{m}.$ In each of these cases, if $(v,w)$ is homoclinic then $|w|$ is
bounded by $\bar{w}$. This follows from the definition of $\bar{w}$ in Lemma
\ref{lem14}, the results of which also apply to (\ref{11}) and (\ref{12a}),
with the same proofs. If $|w|$ exceeds $\bar{w}$ then $p$ is not bounded.

Recall that in Lemmas \ref{lem1} and \ref{lem2}, $r_{0,c}=\left(
u_{0,c},v_{0,c},w_{0,c}\right)  $ denoted the unique solution on the unstable
manifold $\mathcal{U}_{0,c}$ such that $u_{0,c}\left(  0\right)  =u_{m}$ and
$w_{0,c}>0$ on $(-\infty,0].$ In the rest of this proof we will denote this
solution by $\left(  u,v,w\right)  .$ By Lemma \ref{lem15} we can choose $c$
so small that if $t_{1}$ is the first zero of $w_{c},$ then
\begin{equation}
v-\hat{c}<u<v \label{13}%
\end{equation}
on $(-\infty,t_{1}].$

\begin{lemma}
\label{lem13}Condition \ref{13} implies that if $w_{c}>0$ on $(-\infty,t],$
then $\left(  v_{c}\left(  t\right)  ,w_{c}\left(  t\right)  \right)  $ is in
the annular region between the orbit of $\left(  v_{1},w_{1}\right)  $ and the
orbit of $\left(  v_{2},w_{2}\right)  .$
\end{lemma}

\begin{proof}
The proof is similar to the proof of Lemma \ref{lem12}. As long as $w_{c}>0,$
$u_{c}$ and $w_{c}$ can be considered functions of $v_{c}.$ Also,
\begin{equation}
\frac{dw_{c}}{dv_{c}}=\frac{b^{2}\left(  v_{c}-q_{0}S\left(  u_{c}\right)
\right)  }{w_{c}}. \label{14}%
\end{equation}
By considering the eigenvalues of the linearizations of (\ref{10}) as
functions of $c$ we can show, using (\ref{17a}), that for large negative $t,$
$\left(  v\left(  t\right)  ,w\left(  t\right)  \right)  $ lies in the claimed
annular region. Suppose that for some first $\tau,$ $\left(  v_{c}\left(
\tau\right)  ,w_{c}\left(  \tau\right)  \right)  $ lies on the upper boundary
of this region, that is, on the homoclinic orbit of (\ref{11}), at a point
where $w>0.$ \ The slope of this homoclinic orbit at this point is
\[
\frac{dw_{2}}{dv_{2}}=\frac{b^{2}\left(  v_{2}-q_{0}S\left(  v_{2}-\hat
{c}\right)  \right)  }{w_{2}}=\frac{b^{2}\left(  v_{c}\left(  \tau\right)
-q_{0}S\left(  v_{c}\left(  \tau\right)  -\hat{c} \right)  \right)  }%
{w_{c}\left(  \tau\right)  }.
\]
But $u_{c}\left(  \tau\right)  >v_{c}\left(  \tau\right)  -\hat{c} $ (since
$u_{c}<v_{c}$ as long as $w_{c}\geq0$), and since $S$ is increasing and
$w_{c}\left(  \tau\right)  >0$ , it follows from (\ref{14}) that $\frac
{dw_{c}}{dv_{c}}<\frac{dv_{2}}{dw_{2}},$ and so the curve $\left(  v_{c}%
,w_{c}\right)  $ arrives at this point from outside of the annular region,
contradicting the definition of $\tau$. In a similar manner it is shown that
$\left(  v_{c},w_{c}\right)  $ lies above the orbit of $\left(  v_{1}%
,w_{1}\right)  $ as long as $w_{c}>0.$ This uses the bound $u_{c}<v_{c}$ as
long as $w_{c}>0$.

A similar comparison shows that if $t=t_{1}\left(  c\right)  $ is the first
point where $w_{c}\left(  t\right)  =0$ then $v_{c}\left(  t_{1}\left(
c\right)  \right)  $ is an increasing function of $c,$ for $0<c<c_{0}^{\ast},$
and that for $c>c_{0}^{\ast},$ $w_{c}>0$ on $R.$ \ This shows the uniqueness
of $c_{0}^{\ast}.$ To complete the proof of Lemma \ref{lem2} we show that from
the first $t_{1}$ where $w_{c}\left(  t_{1}\right)  =0,$ the curve $\left(
v_{c},w_{c}\right)  $ lies either to the right or below the orbit $\left(
v_{3},w_{3}\right)  ,$ and also below the left branch of the unstable manifold
of (\ref{12a}), at least up to the point where $w_{c}=-\bar{w}$ ($=-\sqrt{2}b$
). (If $w_{c}=-\bar{w},$ then, as in Lemma \ref{lem14}, $v_{c}$ becomes
negative, which is what we are trying to show. See Figure 8.) This follows by
the same sort of comparison as above, now comparing $\left(  v_{c}%
,w_{c}\right)  $ with the lower half of the unstable manifold of (\ref{12a}).
This is possible because by Lemma \ref{lem15}, $u_{c}<v_{c}+\hat{c}$ as long
as $-\bar{w}<w_{c}<0$.

To prove that $c_{0}^{\ast}$ as defined in Lemma \ref{lem2} exists, we note
that the set of $c$ \ such that $w\left(  s_{1}\right)  <0$ for some $s_{1}$
is open, as is the set of $c$ such that $v\left(  s_{2}\right)  >1$ for some
$s_{2}$ and $v>0$ on $(-\infty,s_{2}]$ . This follows because $p_{c}$ is a
continuous function of $c.$ Lemmas \ref{lem13a} and \ref{lem13} imply that
these sets are nonempty, and their definitions and Proposition \ref{prop1}
imply that they are disjoint. Since the interval $\left(  0,\infty\right)  $
is connected, the existence of some positive $c_{0}^{\ast}$ which is not in
either set.\ \ Its uniqueness follows from the Corollary to Lemma 12.

From the definition of $c_{0}^{\ast},$ $w_{c_{0}^{\ast}}\geq0$ on $R.$ Suppose
that there is an $s$ with $w_{c_{0}^{\ast}}\left(  s\right)  =0$ and
$w_{c_{0}^{\ast}}\left(  t\right)  >0$ on $\left(  -\infty,s\right)  .$ Then
$w_{c_{0}^{\ast}}^{\prime}\left(  s\right)  =0,$ $w_{c_{0}^{\ast}}%
^{\prime\prime}\left(  s\right)  \geq0$ and
\[
w_{c_{0}^{\ast}}^{\prime\prime}\left(  s\right)  =-b^{2}q_{0}S^{\prime}\left(
u_{c_{0}^{\ast}}\left(  s\right)  \right)  u_{c_{0}^{\ast}}^{\prime}\left(
s\right)  .
\]
Because $v_{c_{0}^{\ast}}^{\prime}>0$ on $\left(  -\infty,s\right)  ,$
$u_{c_{0}^{\ast}}^{\prime}\left(  s\right)  >0,$ giving $w_{c_{0}^{\ast}%
}^{\prime\prime}\left(  s\right)  <0.$ This contradiction completes the proof
of Lemma \ref{lem1}. \ 
\end{proof}

To complete Lemma \ref{lem2} we must prove the assertions in the third and
fourth sentences. \ In the third sentence,
\[
u^{\prime\prime}=\frac{v^{\prime}-u^{\prime}}{c}=\frac{w}{c}%
\]
when $u^{\prime}=0,$ and at the first zero of $u^{\prime},$ $w<0,$ so
$u^{\prime\prime}<0.$ The implicit function theorem and the comparison
(\ref{a1a}) imply the limit statement.

For the last sentence of Lemma \ref{lem2}, it suffices to prove that
$u_{0,c}\left(  T\left(  c\right)  \right)  >u_{knee.}$ Suppose instead that
$u_{0,c}\left(  T\left(  c\right)  \right)  \leq u_{knee.}$ Since $c\geq
c_{1},$ Lemma \ref{lem12} and the hypotheses of Lemma \ref{lem2} imply that
\begin{align*}
v_{0,c}\left(  T\left(  c\right)  \right)   &  >v_{0,c_{1}}\left(  T\left(
c_{1}\right)  \right)  >q_{0}S\left(  u_{knee}\right) \\
&  \geq q_{0}S\left(  u_{0,c}\left(  T\left(  c\right)  \right)  \right)  .
\end{align*}
Hence $w^{\prime}\left(  T\left(  c\right)  \right)  >0,$ which contradicts
the definition of $T\left(  c\right)  $. This completes the proof of Lemma
\ref{lem2}.

\subsection{Proof of Lemma \ref{lem10}}

\begin{proof}
This result is about system (\ref{1}). However the argument in Lemma
\ref{lem14}, initially about system (\ref{2}), applies equally well to
(\ref{1}), so if $w_{c}$ increases monotonically to above $\bar{w}$ then $v$
crosses $1,$ followed by $u.$ Hence we can assume that if $s_{1}$ is the first
zero, if any, of $w_{c}$, then $w_{c}<\bar{w}$ on $(0,s_{1})$. As earlier in
obtaining (\ref{13}), it follows that if $p=p_{c}$, then
\begin{equation}
v>u>v-c\bar{w} \label{2b}%
\end{equation}
on $(0,s_{1})$. Therefore,
\begin{equation}
\lim_{c\rightarrow0^{+}}\left(  u-v\right)  =0 \label{x1}%
\end{equation}
uniformly on $(0,s_{1}]$ and for $\varepsilon>0$.

Also,
\begin{equation}
\lim_{\frac{\varepsilon}{c}\rightarrow\infty}\left(  q-\frac{1}{1+\beta
S\left(  u\right)  }\right)  =0, \label{x2}%
\end{equation}
uniformly on $(-\infty,s_{1}]$ and for $c>0,$ $\varepsilon>0$. This is proved
by the same argument which lead to (\ref{x1}), .

Now consider the equation obtained from (\ref{1}) by formally setting $c=0$ in
(\ref{1}), namely
\begin{equation}
v^{\prime\prime}=b^{2}Z\left(  v\right)  , \label{2a}%
\end{equation}
where
\[
Z\left(  v\right)  =v-\frac{S\left(  v\right)  }{1+\beta S\left(  v\right)
}.
\]

Because (\ref{1}) has only one equilibrium point, $Z\left(  v\right)  >0$ if
$v>0.$ Hence there is a $\hat{t}>0$ such that if for some $t,$ $u_{m}\leq
v\left(  t\right)  <1$ and $v^{\prime}\left(  t\right)  >0$, then $v\left(
t+\hat{t}\right)  >1.$ (Here $\hat{t}$ is independent of the particular
solution involved.) Lemma \ref{lem10} then follows from (\ref{x1}) and
(\ref{x2}).
\end{proof}

\section{Proof of Lemma \ref{lem2b}}

\begin{proof}
Suppose that the linearization of (\ref{1}) around $p_{0}$ is $P^{\prime}=BP.$
Then
\begin{equation}
B=\left(
\begin{array}
[c]{rrrr}%
-\frac{1}{c} & \frac{1}{c} & 0 & 0\\
0 & 0 & 1 & 0\\
-b^{2}q_{0}S^{\prime}\left(  u_{0}\right)  & b^{2} & 0 & -b^{2}S\left(
u_{0}\right) \\
-\frac{\varepsilon}{c}\beta q_{0}S^{\prime}\left(  u_{0}\right)  & 0 & 0 &
-\frac{\varepsilon}{c}\left(  1+\beta S\left(  u_{0}\right)  \right)
\end{array}
\right)  . \label{16}%
\end{equation}
The characteristic polynomial of $B$ is%
\begin{align}
&  g(X)=X^{4}+\frac{1}{c}\left(  1+\varepsilon\left(  \beta S\left(
u_{0}\right)  +1\right)  \right)  X^{3}+\left(  -b^{2}+\frac{1}{c^{2}%
}\varepsilon\left(  \beta S\left(  u_{0}\right)  +1\right)  \right)
\allowbreak X^{2}\\
&  +\frac{b^{2}}{c}\left(  q_{0}S^{\prime}\left(  u_{0}\right)  -1-\varepsilon
\left(  \beta S\left(  u_{0}\right)  +1\right)  \right)  X\nonumber\\
&  +\frac{b^{2}}{c^{2}}\varepsilon\left(  q_{0}S^{\prime}\left(  u_{0}\right)
-1-\beta S\left(  u_{0}\right)  \right) \nonumber
\end{align}

While proving Lemma \ref{lem1} we showed that if $\varepsilon=0,$ then one of
the non-zero eigenvalues of $B$ is positive and two are real and negative. We
also saw that $q_{0}S^{\prime}\left(  u_{0}\right)  <1,$ and therefore, $\det
B\ <0$ if $\varepsilon>0.$ Since the trace of $B$ is also negative, if
$\varepsilon>0$ then $B$ has either one or three eigenvalues with negative
real part, and for sufficiently small $\frac{\varepsilon}{c}$ it has three,
all of which are real. In fact, since $g\left(  0\right)  <0$, $g^{\prime
}\left(  0\right)  <0$, and $g^{\prime\prime\prime}\left(  X\right)  >0$ if
$X>0,$ $B$ has exactly one real positive eigenvalue for every $\left(
\varepsilon,c\right)  $ in the positive quadrant $\varepsilon>0,c>0.$ For each
$c>0,$ as $\varepsilon$ increases the other roots of $g$ remain in the left
hand plane unless, for some $\varepsilon,$ two of them are pure imaginary.
\ Consideration of the characteristic polynomial in this case (one negative,
one positive, and two pure imaginary roots) shows that the coefficients of $X$
and $X^{3}$ have the same sign. \ This is not the case with $g,$ because the
coefficient of $X^{3}$ is positive and the coefficient of $X$ is negative.

Hence, as asserted in Lemma \ref{lem2b}, the unstable manifold $\mathcal{U}%
_{\varepsilon,c}$ of (\ref{1}) at $p_{0}$ is one dimensional. Further, because
$q_{0}S^{\prime}\left(  u_{0}\right)  <1,$ it follows from (\ref{16}) that if
$\mathbf{\mu=}\left(  \mu_{1},\mu_{2},\mu_{3},\mu_{4}\right)  $ is the unit
eigenvector of $B$ with $\mu_{1}>0,$ then $\mu_{2}>0,$ and $\mu_{3}>0.$ Also
(\ref{16}) implies that if $\varepsilon=0$ then $\mu_{4}=0$ and if
$\varepsilon>0$ then $\mu_{4}<0.$ The claimed behavior for large negative $t$
of solutions on $\mathcal{U}_{\varepsilon,c}$ follows. The continuity of
$\mathcal{U}_{\varepsilon,c}$ for $\varepsilon\geq0$ follows from Theorem 6.1
in chapter 6 in the text of Hartman, \cite{hartman}. \footnote{Our terminology
is different from Hartman's because we define stable and unstable manifolds
even when $\varepsilon=0.$ In this case the unstable manifold $\mathcal{U}%
_{\varepsilon,c}^{+},$ all we need,\ is the set of all solutions $p\left(
t\right)  $ which tend to $p_{0}$ at an exponential rate as $t\rightarrow
-\infty.$ To obtain the desired continuity of $\mathcal{U}_{\varepsilon,c}%
^{+}$ with respect to $\varepsilon$ and $c,$ apply Hartman's theorem to
(\ref{2}) augmented with equations $\varepsilon^{\prime}=0$ \ and $c^{\prime
}=0.$ This is the closest we come to center manifolds in our approach.}

The final assertion of the lemma, that $\lambda_{1}\left(  c,\varepsilon
\right)  >\lambda_{1}\left(  c,0\right)  $ if $c>0$ and $\varepsilon>0$
follows by writing the characteristic polynomial of $B$ in the form
\[
g\left(  x\right)  =Xf\left(  X\right)  +\frac{\varepsilon}{c}\left(  1+\beta
S\left(  u_{0}\right)  \right)  f\left(  X\right)  -\frac{\varepsilon}{c}%
b^{2}S\left(  u_{0}\right)  \beta S^{\prime}\left(  u_{0}\right)  q_{0}.
\]
We see that since $f(\lambda_{1}\left(  c,0\right)  )=0$, $g(\lambda
_{1}\left(  c,0\right)  )<0.$ Hence $\lambda_{1}\left(  c,\varepsilon\right)
>\lambda_{1}\left(  c,0\right)  ,$ completing the proof of Lemma
\ \ \ref{lem2b}.
\end{proof}

\end{document}